\documentclass[10pt]{article}

\usepackage{amssymb}
\usepackage{graphicx}%
\usepackage{multirow}%
\usepackage{amsmath,amssymb,amsfonts}%
\usepackage{amsthm}%
\usepackage{mathrsfs}%
\usepackage[title]{appendix}%
\usepackage{xcolor}%
\usepackage{textcomp}%
\usepackage{manyfoot}%
\usepackage{booktabs}%
\usepackage{algorithm}%
\usepackage{algorithmicx}%
\usepackage{algpseudocode}%
\usepackage{listings}%
\usepackage{url}
\usepackage{fancyhdr}
\usepackage{enumitem} % pour pouvoir donner la forme des bullets dans itemize
\usepackage[a4paper,margin=2cm]{geometry}
\usepackage{authblk}
\usepackage[colorlinks=true,%linkcolor=green,
urlcolor=red, pdfauthor=DJAOUE Séraphin, pdftitle= Global analysis an Age-structured model of Malaria]{hyperref}

\newtheorem{theorem}{Theorem}
\newtheorem{proposition}{Proposition} 
\newtheorem{remark}{Remark}
\newtheorem{lemma}{Lemma}
\newtheorem{assumption}{Assumption}
\renewenvironment{proof}{{\bfseries Proof. \mbox{ }}}{\qed}
\newtheorem{definition}{Definition}

\newcommand{\R}{\mathbb{R}}
\renewcommand{\S}{\mathcal{S}}
\newcommand{\C}{\mathbb{C}}
\newcommand{\RR}{\mathcal{R}}
\renewcommand{\L}{\mathcal{L}}

\newcommand{\N}{\mathbb{N}}

\newcommand{\ess}{\textnormal{ess}}
\newcommand{\X}{\mathcal{X}}
\newcommand{\ep}{\varepsilon}
\newcommand{\Ro}{\mathcal{R}_0}
\newcommand{\Co}{\mathcal{C}}

\begin{document}
		 \date{}
		\title{Global stability in an age-structured SIRS malaria transmission model}
		\author[1]{S.~Djaoue}%\corref{cor1}
		\affil[1]{Department of Mathematics and Computer Science, Faculty of Science, University of Maroua, Cameroon.}
				
		\author[2,*]{Q.~Richard}
		\affil[2]{Institut Montpellierain Alexander Grothendieck, CNRS, University of Montpellier, 34090 Montpellier, France.}
		\affil[*]{Corresponding author: quentin.richard@umontpellier.fr}
		
		\author[3]{A.~Perasso}
		\affil[3]{UMR CNRS 6249 Chrono-environnement, University Bourgogne Franche-Comt\'{e}, France.}
		
		\author[4]{I.~Damakoa}
		\affil[4]{Department of Mathematics and Computer Sciences, Faculty of Science, University of Ngaoundere, Cameroon.}

\maketitle

\hrule 
\begin{abstract}
This paper proposes and analyzes a malaria transmission model structured by the chronological age of the human host population. The model couples an age-structured SIRS system for humans, incorporating waning immunity, with an SI system for mosquitoes under mass-action transmissions. Using integrated semigroup theory and spectral analysis, we establish the well-posedness of the model, derive the basic reproduction number, and prove the global asymptotic stability of the parasite-free equilibrium by using a Lyapunov functional, when $\Ro\leq 1$, thereby excluding the possibility of backward bifurcation. Numerical simulations further suggest the global stability of the endemic equilibrium when $\Ro>1$.
\end{abstract}

 	
\vspace{0.3cm}
 	    
\hrule 

\section{Introduction}
\label{sec:Introduction}

Malaria remains the most widespread and deadly parasitic disease worldwide, with 263 million cases and more than 597,000 deaths reported in 2023 \cite{WHO2024}. It continues to pose a major public health challenge, particularly in the WHO African Region, which bears the greatest burden of the disease. In 2023, this region accounted for an estimated $94\%$ of malaria cases and $95\%$ of deaths globally, with $76\%$ of these deaths occurring among children under the age of five \cite{hviid04,WHO2024}. Although children under five and pregnant women are the most vulnerable groups, recent studies show that older children and young adults (aged 5-30) have a higher infection prevalence and constitute the largest human reservoir for mosquitoes \cite{coalson18, felger12}. These observations highlight the importance of incorporating host age structure into malaria transmission models in order to capture more accurately the underlying epidemiological dynamics.

Mathematical modelling of malaria transmission dates back to the pioneering work of Ross in 1911 \cite{ross1911}. Since then, numerous models have been developed to study this vector-borne disease, integrating a variety of parameters to better understand transmission dynamics and improve control strategies (see \cite{Brauer2019,Inaba2017,LiMartcheva2020} and references therein). Most existing models rely on ordinary differential equations (ODEs) \cite{cai17, chitnis06, djidjou19,Esteva2009,li15}, difference equations \cite{li13, li19} or delay differential equations that account for incubation periods \cite{cooke79,ruan08, vargas12}. Some ODE-based frameworks stratify populations into discrete age classes \cite{addawe12, beretta18, Forouzannia2014,Forouzannia2015}, whereas more recent approaches use partial differential equations (PDEs) to represent continuous age structure in the host population \cite{vargas14,vogtgeisse12, vogtgeisse13} or in the vector population \cite{bellan10,Styer2007}. Chronological age can also be incorporated into within-host models \cite{aguas08,tumwiine08}. In addition, several models are structured by infection age (see for example \cite{wang18, wang20}), which can further refine the representation of within-host dynamics (see \cite{cai17b, djidjou13}). Finally, models that simultaneously incorporate both chronological age and infection-age structures have been proposed in \cite{Richard21,Richard24}.

In this paper, we propose an age-structured model describing the transmission of malaria parasites between mosquitoes and humans, in which the continuous structuring variable is the chronological age of the human population. The model is formulated as a SIRS system for the host population, accounting for waning immunity \cite{keegan13} and as a SI system for the vector population with mass-action forces of infection. The model is described in the next section and is given by:
\begin{equation}
\left\{ \begin{array}{lll}
(\partial_t + \partial_a) s_h(t,a) & =  -\lambda_h(t,a)s_h(t,a)- \mu_h(a) s_h(t,a)+ r_2(a) r_h(t,a)  \\ 
(\partial_t + \partial_a) i_h(t,a) & = \lambda_h(t,a)s_h(t,a) - (\mu_h(a)+\delta(a)+r_1(a))i_h(t,a)\\
(\partial_t + \partial_a) r_h(t,a) & = r_1(a) i_h(t,a) - (r_2(a)+\mu_h(a)) r_h(t,a) \\ 
S_v'(t) & = \Lambda_v  - \lambda_v(t)S_v(t) - \mu_v S_v(t) \\
I_v'(t) & = \lambda_v(t)S_v(t) - \mu_v I_v(t), \\
(s_h,i_h,r_h)(t,0) & = (\Lambda_h,0,0).
\end{array}
\right.
\label{system}
\end{equation}
for every $t>0$ and $a\geq 0$, with the following initial conditions
\begin{equation}
\left\{ \begin{array}{lll}
(s_h,i_h,r_h)(0,a) & = (s_{h,0},i_{h,0},r_{h,0})(a)  \\ 
(S_v,I_v)(0) & = (S_{v,0},I_{v,0})
\end{array}
\right.
\label{condition}
\end{equation}
and the following forces of infection
\begin{equation}
\lambda_h(t,a)  = \beta_v(a)I_v(t), \qquad \lambda_v(t)  = \int_0^{+\infty} \beta_h(a) i_h(t,a) da.
\label{eq3}
\end{equation}
After integration w.r.t. $a$, the model \eqref{system} rewrites as
\begin{equation}
\left\{ \begin{array}{lll}
S_h'(t)& = \Lambda_h -\beta_v I_v(t)S_h(t)- \mu_h S_h(t)+ r_2 R_h(t)  \\ 
I_h'(t)& = \beta_v I_v(t)S_h(t) - (\mu_h+\delta+r_1)I_h(t)\\
R_h'(t) & = r_1 I_h(t) - (r_2+\mu_h) R_h(t) \\ 
S_v'(t) & = \Lambda_v  - \beta_h S_v(t) I_h(t) - \mu_v S_v(t) \\
I_v'(t) & = \beta_h S_v(t)I_h(t)- \mu_v I_v(t).
\end{array}
\right.
\label{Eq:EDO}
\end{equation}
The ODE model \eqref{Eq:EDO} was previously investigated in \cite{Garba2008} in the case $r_2=0$, \textit{i.e.} in the absence of reinfection of recovered humans. In that work, the authors first examined the case where the forces of infection are ratio dependent (that is, $\lambda_h$ and $\lambda_v$ in \eqref{eq3} are proportional to the fraction of infected humans or mosquitoes relative to the total population). They proved the existence of a backward bifurcation, meaning that two endemic equilibria may exist when the basic reproduction number $\Ro$ is close and slightly below one. They then considered the case where the forces of infection follow a mass-action formulation, showing that only a forward bifurcation can occur: a unique endemic equilibrium exists only when $\Ro>1$, while the parasite-free equilibrium is globally asymptotically stable (GAS) for $\Ro\leq 1$.

This line of investigation was extended in \cite{niger08} to the case $r_2\neq 0$. In a more general model incorporating repeated-exposure immunity, the authors showed that in the ratio-dependent case, both their general model and \eqref{Eq:EDO} (see also \cite{Mukandavire2009}) exhibit backward bifurcations. In contrast, in the mass-action case, backward bifurcation may occur in the general model (where multiple infections are possible), whereas only forward bifurcation arises for \eqref{Eq:EDO}. They also proved that the parasite-free equilibrium is locally asymptotically stable (LAS) if $\Ro<1$.

It is worth noting that in \cite{niger08} a proof of the global stability of the endemic equilibrium is proposed in the case $r_2=0$. However, in the computations of the derivative of the Lyapunov candidate function, it is implicitly supposed that the forces of infection are constant while they should depend on both susceptible and infected populations. As a result, some terms cannot be cancelled as claimed.

In \cite{cai2013}, a similar SIR/SI model (corresponding to \eqref{Eq:EDO} with $r_2=0$) is analysed with an additional reinfection term allowing recovered humans to return to the infected compartment. The authors show that backward bifurcation occurs in the ratio-dependent case, whereas only forward bifurcations appears in the mass-action case.

In \cite{Roop2015}, the authors investigated a related SEIRS/SEI model with various incidence functions. They established the existence of backward bifurcation in the ratio-dependent case, while only forward bifurcation arises for Holling type I (which corresponds to mass-action) and type II incidence functions. The same study also proposes a proof of the GAS of the endemic equilibrium under strong assumptions. However, in the computation of the derivative of the Lyapunov candidate function, arguments of the form ``if $S_h<S_h^{**}$ and $E_h>E_h^{**}$ then $1-\frac{S_hE_h^{**}}{S_h^{**}E_h}>0$" are used to establish the non-positivity of the derivative. Such arguments are not valid unless it is first shown that the solution remains within an appropriate positively invariant set, which is not demonstrated in the paper.

In \cite{li15}, another related SIR/SI model (which reduces to \eqref{Eq:EDO} with $r_2=0$ for particular parameter values) is studied with the addition of vaccination compartments. Once again, the authors establish the presence of backward bifurcation in the ratio-dependent case but not in the mass-action case. They also analyse the stability properties of the equilibria, proving the GAS of the parasite-free equilibrium when $\Ro<1$ and the LAS of the endemic equilibrium when $\Ro>1$.

All other references mentioned above consider models that either employ ratio-dependent forces of infection \cite{addawe12,cai17,chitnis06,djidjou19,Esteva2009,Forouzannia2014,Forouzannia2015,Richard21,Richard24,vogtgeisse12,vogtgeisse13} , which typically lead to backward bifurcation, or study SIR/SI-type models without reinfection of recovered humans \cite{cai17b,vargas14,wang20}, or assume non-constant immigration rates \cite{beretta18,wang18}. See also \cite{mandal2011} for a review of the various compartmental models proposed in the literature.

To summarize, two properties were previously established for the ODE model: the local asymptotic stability of the parasite-free equilibrium and the existence of only a forward bifurcation \cite{niger08}. The analysis of the PDE model \eqref{system} presented here is new and involves two main difficulties: first, the structured of the model itself (SIRS/SI), and second, the age-structure of the human population. Indeed, in \cite{OkuwaInaba2021}, the authors studied an age-structured SIRS model without vector and left open the question of the uniqueness of the endemic equilibrium when $\Ro>1$. In the present paper, we mainly prove the global asymptotic stability of the parasite-free equilibrium when $\Ro\leq 1$ thereby ruling out the possibility of backward bifurcation, and extending existing results. 

In section 2, we present the PDE model \eqref{system} and establish the framework required to prove the well-posedness of a global positive solution using the theory of integrated semigroups. In Section 3, we first derive the basic reproduction number $\RR_0$ by means of the next-generation operator. We then linearise the system around each equilibrium, obtaining linear $\mathcal{C}_0$-semigroups, and apply spectral theory to determine the local stability of the parasite-free equilibrium when $\RR_0<1$ and its instability when $\RR_0>1$. In this section, we also employ classical results from dynamical systems to study the attractiveness of the parasite-free equilibrium when $\RR_0\leq 1$, using a Lyapunov functional defined on certain omega-limit sets. This Lyapunov functional further allows us to establish the global stability of this equilibrium in the case $\RR_0=1$. This section concludes with the existence and uniqueness of the endemic equilibrium in the ODE case and in the PDE case when $r_2\equiv 0$. Under these two assumptions, we prove the local asymptotic stability of the endemic equilibrium. Finally, Section 4 presents several numerical simulations for the PDE model, which suggest the global asymptotic stability of the endemic equilibrium when $\RR_0>1$.

\section{Well-posedness and global solution}
\label{sec:existence}

\subsection{Model formulation}
\label{sec:modelFormulation}

As presented in the introduction, in this paper we focus on the age-structured mathematical model of malaria transmission with waning immunity described by \eqref{system}. As humans might be repeatedly infected because they have not acquired permanent immunity, it is assumed that the human population is described by the SIRS (Susceptible-Infected-Recovered-Susceptible) equations. Recovered human hosts have temporary immunity that can be lost and are again susceptible to reinfection. On their side, mosquitoes are assumed not to recover from the parasites infection due to their short lifespan. Consequently the mosquito population is described by the SI equations. All newborns are susceptible to infection, and the development of malaria starts when the infectious female mosquito bites the human host. The vectors do not die from the infection or are otherwise harmed. In the model, the parameters have the following biological meaning:
\begin{itemize}
\item $S_v(t)$ and $I_v(t)$ respectively, denote the number of susceptible and infectious mosquitoes at time $t \geq 0$.
\item $s_h(t,a)$, $i_h(t,a)$ and $r_h(t,a)$ respectively denote the number of susceptible, infectious and recovered humans of age $a\geq 0$ at time $t$. Also for any $ a_1, a_2 \in (0, + \infty) $, $ a_1 <a_2 $, the quantity $\int_{a_1}^{a_2} s_h(t,a) da $ denotes the total number of susceptible individuals at time $t$ of ages between $a_1$ and $a_2$. In particular $\int_0^{+ \infty} s_h(t,a)da $ is the total density of susceptible humans at time $t$ and
$$N_h(t)=\int_0^\infty (s_h(t,a)+i_h(t,a)+r_h(t,a))da, \qquad N_v(t)=S_v(t)+I_v(t)$$
are respectively the total population of humans and mosquitoes at time $t$.

\item $\lambda_h(t,a)$ is the age-dependent infection rate of susceptible humans (following contact with infectious mosquitoes) of age $a$ at time $t$ defined by \eqref{eq3} where $\beta_v(a)$ is the parasite transmission rate from mosquitoes to humans of age $a$.
\item Susceptible adult mosquitoes acquire infection, following effective contact with an infectious human (via a blood meal), at a rate $\lambda_v(t)$, given by \eqref{eq3} where $\beta_h(a)$ is the parasite transmission rate from humans of age a to mosquitoes. 
\item $\Lambda_h$ and $\Lambda_v$ respectively denote the human and the mosquito recruitment rates while $\mu_v$ denotes the natural death rate of mosquitoes.
\item The age-dependent functions $\mu_h,\delta, r_1$ and $r_2$ are respectively the natural death rate, per capita parasite-induced death rate, per capita recovery rate and the per capita rate of loss immunity for human individual.
\end{itemize}
We end the description of the model with the following assumption.
\begin{assumption}\label{Assum:general}
We suppose that:
\begin{enumerate}
\item There exists $\mu_0>0$ such that $\mu_h(a) \geq \mu_0>0$ for almost every $a\geq 0$ and $\mu_v\geq \mu_0$,
\item $\Lambda_h>0$, $\Lambda_v>0$, $\mu_v>0$,
\item $\beta_v,\beta_h, \mu_h,\delta, r_1, r_2 \in L_+^{\infty}(0,+\infty)$.
\end{enumerate}
\end{assumption}
In all that follows we suppose that Assumption \ref{Assum:general} holds.

\subsection{Abstract Cauchy problem}

To prove the existence of solutions of the system \eqref{system}-\eqref{condition}, we follow \cite{magal2018} and use integrated semigroup theory whose approach was introduced by Thieme \cite{thieme90}. We first put the system \eqref{system} in the form of an abstract Cauchy problem.

In all that follows, let us define the Banach space $X= \mathbb{R} \times \left[L^1(0,\infty)\right]^3 \times \mathbb{R}^2$ endowed with the product norm 
\[ \left\Vert \left( \begin{array}{c}
\kappa \\ s_h \\ i_h \\r_h \\S_v \\ I_v
\end{array} \right) \right\Vert_X = |\kappa|+||s_h||_{L^1} + ||i_h||_{L^1}+||r_h||_{L^1} + |S_v|+|I_v| \]
and $X_+$ is the non-negative cone of $X$ that is $X_+ = \mathbb{R}_+ \times \left[L_+^1(0,\infty)\right]^3 \times \mathbb{R}_+^2$. For every constant $r>0$ and $x\in X$, we denote by $B_X(x,r)$ the ball of $X$ centred in $x$ and with radius $r$:
\begin{equation}\label{Eq:ball_r}
B_X(x,r) = \{y\in X: ||y-x||_X \leq r\}.
\end{equation}
Let $A: D(A) \subset X \rightarrow X$ be the linear operator defined by
\[ A \left( \begin{array}{c} 0 \\ s_h \\ i_h \\ r_h \\ S_v \\ I_v \end{array} \right)  = 
\left( \begin{array}{c} -s_h(0)\\ - s_h' - \mu_h  s_h  \\- i_h'- (\mu_h+\delta+r_1) i_h\\- r_h' - (r_2+\mu_h )r_h\\ -\mu_v S_v \\ - \mu_v I_v \\ \end{array}\right) \]
with domain 
\[ D(A) = \{0\}\times \left\{ (s_h,i_h,r_h) \in \left[ W^{1,1} (0,\infty)\right]^3  :  i_h (0) = r_h(0) = 0 \right\} \times \mathbb{R}^2. \]
One may note that $\overline{D(A)}= \{0\} \times(L^1(0,\infty))^3 \times \mathbb{R}^2$ and is not dense in $X$. Let $F : \overline{D(A)} \rightarrow X$ be the non-linear function defined by
\begin{equation}
F \left( \begin{array}{c} 0 \\ s_h \\ i_h \\ r_h\\ S_v \\ I_v \end{array} \right) = 
\left( \begin{array}{c} 
\Lambda_h \\ -\beta_v s_h I_v + r_2 r_h\\ 
\beta_v s_h I_v \\
r_1 i_h \\ 
\Lambda_v -S_v\int_0^{+\infty} \beta_h(a) i_h(a) da  \\ 
S_v\int_0^{+\infty} \beta_h(a) i_h(a) da \end{array} \right).
\label{eq4}
\end{equation}
We define $X_0 = \overline{D(A)}$ and $X_{0+} = X_0 \cap X_+ = \{0\} \times(L_+^1(0,\infty))^3 \times (\mathbb{R}_+)^2  $. Now by identifying $(0, s_h(t,.), i_h(t,.), r_h(t,.), S_v(t), I_v(t))$ to $\displaystyle x(t)$, the equation \eqref{system} can be rewritten as the following abstract semilinear Cauchy problem in $X_0$
\begin{equation}
\left\{ \begin{array}{rcl}
 \dfrac{d x(t)}{d t}  &=& A x(t) + F(x(t)), \ \forall t> 0\\
 x(0) &=& (0, s_{h,0},i_{h,0},r_{h,0}, S_{v,0},I_{v,0}) \in  X_{0}.
\end{array}  \right.
\label{pb_cauchy}
\end{equation}

\subsection{Linear and nonlinear part of Cauchy problem}

Here we handle the linear part with the following result.
\begin{proposition}
The operator $A:D(A) \subset X$ is a Hille-Yosida operator with $(-\mu_0,+\infty) \subset \rho(A)$ (which denotes the resolvent set of $A$) and for all  $\lambda > -\mu_0$, $(\lambda I - A)^{-1} X_+ \subset X_+$. Moreover, for all $\omega> 0$, the  operator $A-\omega I$ is also a Hille-Yosida operator.
\label{theorem_hille}
\end{proposition}
\begin{proof}
Let $(\kappa, \phi_1,\phi_2,\phi_3,y_1,y_2) \in X$ and $(0, \psi_1,\psi_2,\psi_3,x_1,x_2) \in D(A)$ such that $(\lambda I - A)(0,\psi_1,\psi_2,\psi_3,x_1,x_2) = \break(\kappa, \phi_1,\phi_2,\phi_3,y_1,y_2)$. We have
\[
\left\{ \begin{array}{ll}
\psi_1(a) & = \kappa e^{-\int_0^a (\mu_h(s)+\lambda)ds} +  \int_0^a \phi_1(s) e^{-\int_s^a (\mu_h(\tau)+\lambda)d \tau} ds \\
\psi_2(a) & = \int_0^a \phi_2(s) e^{-\int_s^a (\mu_h(\tau)+\delta(\tau)+r_1(\tau)+\lambda)d \tau} ds \\
\psi_3(a) & = \int_0^a \phi_3(s) e^{-\int_s^a (\mu_h(\tau)+r_2(\tau)+\lambda)d \tau} ds \\
 x_1 & = \frac{1}{\lambda + \mu_v} y_1 \\
x_2 & = \frac{1}{\lambda + \mu_v}y_2
\end{array}  
\right.\]
for each $a \in (0,+\infty)$. Thus we get
$$ ||(\lambda I - A)^{-1}(\kappa, \phi_1,\phi_2,\phi_3,y_1,y_2)||_X \leq \frac{||(\kappa, \phi_1,\phi_2,\phi_3,y_1,y_2)||_X}{\lambda+\mu_0}.$$
It readily follows that $||(\lambda I - A)^{-n}||_{\mathcal{L}(X)} \leq \frac{1}{(\lambda + \mu_0)^n}$ for every $n \in \N $; so that $A$ is a Hille-Yosida operator with $(-\mu_0, +\infty) \subset \rho(A)$. The positivity of the resolvent then simply results from the above resolvent formula.
\end{proof}

Now we treat the nonlinear part by showing a Lipschitz property and a positivity property of $F$.
\begin{proposition}
The function $F:X_0 \rightarrow X$ given in \eqref{eq4} satisfies the following properties:
\begin{enumerate}
\item $F$ is Lipschitz continuous on bounded sets i.e. 

$\forall \ m>0$, $\exists \ c_m>0$ such that  $\forall (w,w') \in (B_{X_0}(0,m) \cap X_0)^2$, $||F(w)-F(w')||_X \leq c_m||w-w'||_X$.
\item $\forall m >0$, $\exists p_m >0$ such that 
\begin{equation*}
w \in B_{X_0}(0,m) \cap X_{0+} \Rightarrow F(w) + p_m w \in X_+. 
\end{equation*}
\end{enumerate}
\label{prop1}
\end{proposition}
\begin{proof}
\begin{enumerate}
\item Let $m >0$ and $w = (0,s_h,i_h,r_h,S_v,I_v)^T$, $w' = (0,s_h',i_h',r_h',S_v',I_v')^T$ be two elements of $B_{X_0}(0,m)$. We have 
$$||F(w)- F(w')||_X  \leq \left( 2m ||\beta_h||_{\infty} + 2m ||\beta_v||_{\infty} + ||r_1||_{\infty} + ||r_2||_{\infty} \right) ||w- w'||_X.$$
With $L(m) = 2m (||\beta_h||_{\infty} + ||\beta_v||_{\infty}) + ||r_1||_{\infty}+||r_2||_{\infty}$, it follows that the first property is satisfied.
\item Let $m>0$ and $w = (0,s_h,i_h,r_h,S_v,I_v)^T \in B_{X_0}(0,m) \cap X_{0+}$. We have 
\[ F(w)+p_m w = 
\left( \begin{array}{c} 
\Lambda_h \\ -\beta_v s_h I_v + r_2 r_h + p_m s_h\\ 
\beta_v s_h I_v + p_m i_h \\
r_1 i_h + p_m r_h\\ 
\Lambda_v -S_v\int_0^{+\infty} \beta_h(a) i_h(a) da +p_m S_v \\ 
S_v\int_0^{+\infty} \beta_h(a) i_h(a) da + p_m I_v\end{array} \right).\]
Since we have 
$$- \beta_v(a) s_h(a) I_v + p_m s_h(a) =  (p_m- \beta_v(a) I_v) s_h(a)\geq (p_m- ||\beta_v||_{\infty} m ) s_h(a) 
$$
for each $a \in (0,\infty)$ and 
$$\Lambda_v -S_v\int_0^{+\infty} \beta_h(a) i_h(a) da + p_m S_v = \Lambda_v + S_v(p_m - \int_0^{+\infty} \beta_h(a) i_h(a) da)\geq \Lambda_v + S_v(p_m - ||\beta_h||_{\infty} m) da)
$$
then it follows that the second property is satisfied once $p_m \geq m(||\beta_h||_{\infty} + ||\beta_v||_{\infty})$.
\end{enumerate}
\end{proof}

We deduce the following result about the existence of a local solution.
\begin{proposition}
For each $x \in X_{0+}$, there exists $t_{\max} \leq +\infty$ and a unique continuous map $U(.)x \in \Co([0,t_{\max}), X_{0+})$ which is an integrated solution of the Cauchy problem \eqref{pb_cauchy} \textit{i.e.} 
\[ \int_0^t U(s)x ds \in D(A), \quad \forall t \in[0,t_{\max}] \]
and
\[U(t)x= x + A\int_0^t U(s)x ds + \int_0^tF(U(s)x)ds, \quad \forall t \in [0,t_{\max}]. \]
\label{prop_existence}
\end{proposition}

\begin{proof}
Using Proposition \ref{theorem_hille} and Proposition \ref{prop1} it follows from \cite[Theorem 5.2.7, p. 226]{magal2018} that there exists a unique local mild solution to the Cauchy problem \eqref{pb_cauchy}. Also, the solution is non-negative due to \cite[Proposition 5.3.2, p. 227]{magal2018}.
\end{proof}

\subsection{Boundedness and global existence}

We define the space $\X=(L^1(\R_+))^3\times \R^2$ and its positive cone $\X_+=(L^1_+(\R_+))^3\times \R^2_+$. We show here that the solution is global in time due to some boundedness properties, which is the main result of this section:

\begin{theorem}
\label{theorem_global}
For every $\widehat{x} = (s_{h,0},i_{h,0},r_{h,0},S_{v,0},I_{v,0}) \in \X_+$, there exists a unique
mild solution $(0,s_h,i_h,r_h,S_v,I_v) \in \Co([0,\infty),X_+)$ that induces a globally defined strongly continuous semiflow via
\[ \Phi: \R_+ \times \X_+ \ni (t,\widehat{x}) \longmapsto \Phi_t(\widehat{x}) = (s_h(t,.),i_h(t,.),r_h(t,.),S_v(t),I_v(t)) \]
where the solution satisfies for each $(t,a)\in \R_+^2$:
\begin{equation}\label{Eq:Estimates_sol1}
s_h(t,a)+i_h(t,a)+r_h(t,a)\leq \Lambda_h e^{-\int_0^a \mu_h(s)ds}\mathbf{1}_{\{a\leq t\}\\
}+(s_{h,0}+i_{h,0}+r_{h,0})(a-t)e^{-\int_{a-t}^a \mu_h(s)ds} \mathbf{1}_{\{a>t\}}
\end{equation}
whence
\begin{equation}\label{Eq:Estimates_sol2}
N_h(t)\leq \max\left\{\dfrac{\Lambda_h}{\mu_0},N_{h}(0)\right\}=:C_h.
\end{equation}
We also have
\begin{equation}\label{Eq:Estimates_sol3}
S_v(t)+I_v(t)=\dfrac{\Lambda_v}{\mu_v}+\left(N_v(0)-\dfrac{\Lambda_v}{\mu_v}\right)e^{-\mu_v t}
\end{equation}
and
\begin{equation}\label{Eq:Estimates_pos}
S_v(t)\geq \dfrac{\Lambda_v}{\mu_v+\|\beta_h\|_{L^\infty}C_h}\left(1-e^{-(\mu_v+\|\beta_h\|_{L^\infty}C_h)t}\right)+S_{v,0}e^{-(\mu_v+\|\beta_h\|_{L^\infty}C_h)t}.
\end{equation}
Moreover, the semiflow can be decomposed as $\Phi_t=(\Phi_t^{s_h}, \Phi_t^{i_h}, \Phi_t^{r_h}, \Phi_t^{S_v}, \Phi_t^{I_v})=(\Phi_t^h, \Phi_t^v)$ and $\Phi_t^h=\Phi_t^{h,1}+\Phi_t^{h,2}$
where:
\begin{equation}\label{Eq:Semiflow_h1}
\begin{array}{rcl}
\Phi_t^{h,1}(\widehat{x})(a)&=&\begin{pmatrix}
s_{h,0}(a-t)e^{-\int_{a-t}^a \mu_h(s)ds}e^{-\int_{a-t}^a \beta_v(s)I_v(t+s-a)ds} \\
i_{h,0}(a-t)e^{-\int_{a-t}^a (\mu_h(s)+\delta(s)+r_1(s))ds}\\
r_{h,0}(a-t)e^{-\int_{a-t}^a (r_2(s)+\mu_h(s))ds}
\end{pmatrix}\mathbf{1}_{\{a\geq t\}} \\
&&+ \begin{pmatrix}
\int_{a-t}^a r_2(s)r_h(t-a+s,s)e^{-\int_s^a (\mu_h(\xi)+\beta_v(\xi)I_v(t-a+\xi))d\xi}ds \\
\int_{a-t}^a \beta_v(s)I_v(t-a+s)s_h(t-a+s,s)e^{-\int_s^a (\mu_h(\xi)+\delta(\xi)+r_1(\xi))d\xi}ds \\
\int_{a-t}^a r_1(s)i_h(t-a+s,s)e^{-\int_s^a (r_2(\xi)+\mu_h(\xi))d\xi}ds 
\end{pmatrix}\mathbf{1}_{\{a>t\}}
\end{array}
\end{equation}
and
\begin{equation}\label{Eq:Semiflow_h2}
\Phi_t^{h,2}(\widehat{x})(a)=\begin{pmatrix}
\Phi_t^{s_h,2}(\widehat{x})(a)\\
\Phi_t^{i_h,2}(\widehat{x})(a)\\
\Phi_t^{r_h,2}(\widehat{x})(a)\end{pmatrix}=\begin{pmatrix}
\Lambda_h e^{-\int_0^a(\mu_h(s)+\beta_v(s)I_v(t-a+s))ds}+\int_0^a r_2(s)r_h(t-a+s,s)e^{-\int_s^a (\mu_h(\xi)+\beta_v(\xi)I_v(t-a+\xi))d\xi}ds\\
\int_0^a \beta_v(s)I_v(t-a+s)s_h(t-a+s,s)e^{-\int_s^a (\mu_h(\xi)+\delta(\xi)+r_1(\xi))d\xi}ds \\
\int_0^a r_1(s)i_h(t-a+s,s)e^{-\int_s^a (r_2(\xi)+\mu_h(\xi))d\xi}ds
\end{pmatrix}\mathbf{1}_{\{t\geq a\}}
\end{equation}
\end{theorem}

\begin{proof}
Let $x = (0,s_{h,0},i_{h,0},r_{h,0},S_{v,0},I_{v,0}) \in X_{0+}$ and set $U(t)x=(0,s_h(t,.),i_h(t,.),r_h(t,.),S_v(t),I_v(t)) \in \Co([0,t_{\max}),X_{0+})$ the solution of \eqref{pb_cauchy}. We see that
\[ N_v'(t) = \Lambda_v - \mu_v N_v(t). \]
A simple use of Gronwall's lemma implies \eqref{Eq:Estimates_sol3} for each $t\in[0,t_{\max})$. Now, note that the components $s_h, i_h$ and $r_h$ of the solution only satisfy \eqref{system} in the mild sense, hence it is not possible to use directly to integrate the equations of system \eqref{system}. However, one may note that the solution is given by
$$s_h(t,a)=\begin{cases} 
s_{h,0}(a-t)e^{-\int_{a-t}^a \mu_h(s)ds}+\int_{a-t}^a (r_2(s)r_h(t-a+s,s)-\lambda_h(t-a+s)s_h(t-a+s))e^{-\int_s^a \mu_h(\xi)d\xi}ds & \text{ if } a\geq t,\\
\Lambda_h e^{-\int_0^a \mu_h(s)ds}+\int_0^a (r_2(s)r_h(t-a+s,s)-\lambda_h(t-a+s)s_h(t-a+s))e^{-\int_s^a \mu_h(\xi)d\xi}ds & \text{ else.}
\end{cases}$$
$$i_h(t,a)=\begin{cases} 
i_{h,0}(a-t)e^{-\int_{a-t}^a \mu_h(s)ds}+\int_{a-t}^a (\lambda_h(t-a+s)s_h(t-a+s)-(r_1(s)+\delta(s))i_h(t-a+s,s))e^{-\int_s^a \mu_h(\xi)d\xi}ds & \text{ if } a\geq t,\\
\int_0^a (\lambda_h(t-a+s)s_h(t-a+s)-(r_1(s)+\delta(s))i_h(t-a+s,s))e^{-\int_s^a \mu_h(\xi)d\xi}ds & \text{ else.}
\end{cases}$$
$$r_h(t,a)=\begin{cases} 
r_{h,0}(a-t)e^{-\int_{a-t}^a \mu_h(s)ds}+\int_{a-t}^a (r_1(s)i_h(t-a+s,s)-r_2(s)r_h(t-a+s,s))e^{-\int_s^a \mu_h(\xi)d\xi}ds & \text{ if } a\geq t,\\
\int_0^a (r_1(s)i_h(t-a+s,s)-r_2(s)r_h(t-a+s,s))e^{-\int_s^a \mu_h(\xi)d\xi}ds & \text{ else.}
\end{cases}$$
Summing all three above equations lead to
$$n_h(t,a)\leq \begin{cases} n_{h,0}(a-t)e^{-\int_{a-t}^a \mu_h(s)ds}\leq n_{h,0}(a-t)e^{-\mu_0 t} & \text{ if } a\geq t,\\
\Lambda_h e^{-\int_0^a \mu_h(\xi)d\xi}ds\leq \Lambda_h e^{-\mu_0 a} & \text{ else.}
\end{cases}$$
where we defined
$$n_h(t,a)=s_h(t,a)+i_h(t,a)+r_h(t,a)$$
and $n_{h,0}(a)=n_h(0,a)$, which implies \eqref{Eq:Estimates_sol1}-\eqref{Eq:Estimates_sol2} for each $t\in[0,t_{\max})$. Finally, since we have
$$S_v'(t)\geq \Lambda_v-S_v(t)\left(\mu_v+\|\beta_h\|_{L^\infty}\|I_h(t,.)\|_{L^1}\right)$$
for each $t>0$, it follows from \eqref{Eq:Estimates_sol2} and a use of Gronwall's lemma that \eqref{Eq:Estimates_pos} holds for each $t\in[0,t_{\max})$. Now, if we suppose by contradiction that $t_{\max}<+\infty$ then, by \cite[Theorem 3.3]{magal2001} we would have 
\begin{equation}
\lim_{t \rightarrow t_{max}} \left(||s_h(t,.)||_{L^1} + ||i_h(t,.)||_{L^1} + ||r_h(t,.)||_{L^1} + |S_v|+|I_v|\right) = +\infty
\label{global1}
\end{equation} 
which contradicts either \eqref{Eq:Estimates_sol2} or \eqref{Eq:Estimates_sol3}, whence $t_{\max} = +\infty$. 

\end{proof}

\section{Equilibria and their Stability}
\label{sec:analysis}

In this section, we handle the existence and the stability of equilibria. Clearly, the parasite-free equilibrium always exists for the model \eqref{system} and is given by
$$E_0 = \left( s_h^0,0,0,S_v^0, 0 \right) \in \X_+$$
where $s_h^0(a) = \Lambda_h \exp(-\int_0^a \mu_h(s)ds)$ for every $a\geq 0$ and $S_v^0=\frac{\Lambda_v}{\mu_v}$. 

\subsection{The basic reproduction number $\RR _0$}

The stability of $E_0$ will depend on the basic reproduction number which is computed in this section. This quantity plays a very important role in epidemiological modelling, see for example \cite{perasso18} for an introduction and some references, and also \cite{Richard21} for its derivation in a similar model. We follow the classical procedure to derive this number, based on the next generation operator \cite{diekmann,Inaba2017}. Let us first define
\begin{equation}\label{Eq:R0}
\RR_0=\sqrt{\dfrac{S^0_v}{\mu_v}\int_0^\infty \int_0^\infty \beta_h(\xi+s)\beta_v(\xi)s^0_h(\xi)e^{-\int_\xi^{\xi+s}(\mu_h(u)+r_1(u)+\delta(u))du}d\xi ds}.
\end{equation}
and make the following
\begin{assumption}\label{Assumption:Omega}
We suppose that:
$$\int_0^\infty \int_0^\infty \beta_h(a+s)\beta_v(a) e^{-\mu_0(a+s)}da ds>0.$$
\end{assumption}
In all that follows we suppose that Assumption \ref{Assumption:Omega} holds. Note that this assumption simply reflects the fact that, at some point during their lifetime, humans are susceptible to infection (\textit{i.e.} $\beta_v(a)>0$ for some $a\in (0,\infty)$) and that they may subsequently transmit the infection to a mosquito (\textit{i.e.} $\beta_h(a+s)>0$ for some $s\geq 0$). We can now state the main result of this section.

\begin{theorem}\label{Thm:R0}
The number $\RR_0$ defined by \eqref{Eq:R0} is the basic reproduction number related to \eqref{system}.
\end{theorem}
\begin{proof}
We begin by linearising the system \eqref{system} in the neighbourhood of the parasite-free equilibrium $E_0$. Hence $(i_h,I_v)$ satisfies the following system:
\begin{equation*}
 \left\{ \begin{array}{rl}
\left(\partial_t+\partial_a\right) i_h(t,a) & =  \beta_v(a) I_v(t) s_h^0(a) -(\mu_h(a)+\delta(a)+r_1(a)) i_h(t,a)  \\
I_v'(t) & = S_v^0 \int_0^{+\infty} \beta_h(a) i_h(t,a)da - \mu_v I_v(t).\\
\end{array} \right.
% \left\{ \begin{array}{ll}
%\lambda_h(a)&  = \beta_v(a) I_v \\
%\lambda_v & =  \int_0^{+\infty} \beta_h(a) i_h(a)da \\
%\end{array}
%\right.
\end{equation*}
Using Volterra's formulation on the first equation of the latter system and a simple integration of the second equation, we get
\begin{equation*}
i_h(t,a) = \left\{ \begin{array}{ll}
i_{h,0}(a-t)e^{-\int_{a-t}^a (\mu_h(s)+\delta(s)+r_1(s))ds} + \int_0^t \beta_v(s+a-t)I_v(s)s_h^0(s+a-t) e^{-\int_{s+a-t}^a(\mu_h(\tau)+\delta(\tau)+r_1(\tau))d \tau}ds & a>t \\
\int_0^a \beta_v(s) I_v(s+t-a) s_h^0(s) e^{-\int_s^a (\mu_h(\tau)+\delta(\tau)+r_1(\tau))d\tau } ds & a \leq  t \end{array} \\ \right.
\end{equation*}
and
\begin{equation*}
I_v(t) = I_v(0)e^{-\mu_v t} + S_v^0 \int_0^t e^{-\mu_v(t-s)} \int_0^{+\infty} \beta_h(a) i_h(s,a)da ds.
\end{equation*}
Let us define $B_h(t,a)$ and $B_v(t)$ the number of newly infected humans of age $a$ and mosquitoes, at time $t$, by
$$B_h(t,a)=\beta_v(a) I_v(t)s^0_h(a), \qquad B_v(t)=S^0_v \int_0^\infty \beta_h(a)i_h(t,a).$$	
It follows that
$$I_v(t)=I_v(0)e^{-\mu_v t}+\int_0^t e^{-\mu_v(t-s)}B_v(s)ds=I_v(0)e^{-\mu_v t}+\int_0^t e^{-\mu_v s} B_v(t-s)ds$$
whence
$$B_h(t,a)=f_{0,h}(t,a)+\beta_v(a)s^0_h(a) \int_0^t e^{-\mu_v s}B_v(t-s)ds$$
%+\beta_v(a)s^0_h(a) \int_0^t e^{-\mu_v s}B_v(t-s)ds$$
where $f_{0,h}(t,a)=I_v(0)\beta_v(a)s^0_h(a)e^{-\mu_v t}$, whence $\|f_{0,h}\|_{L^1}\underset{t\to \infty}{\to}0$. Similarly, we see that 
\begin{equation*}
i_h(t,a)= \left\{ \begin{array}{ll} i_{h,0}(a-t)e^{-\int_{a-t}^a (\mu_h(s)+\delta(s)+r_1(s))ds} + \int_0^t B_h(s,s+a-t) e^{-\int_{s+a-t}^a(\mu_h(\tau)+\delta(\tau)+r_1(\tau))d \tau}ds & a>t \\
\int_0^a B_h(s+t-a,s)e^{-\int_s^a (\mu_h(\tau)+\delta(\tau)+r_1(\tau))d\tau } ds & a \leq  t \end{array} \\ \right.
\end{equation*}
so we deduce, after some computations, that
\begin{flalign*}
B_v(t)=f_{0,v}(t,a)+S^0_v \int_0^t \int_0^\infty \beta_h(u+s)B_h(t-u,s)e^{-\int_s^{s+u}(\mu_h(\xi)+\delta(\xi)+r_1(\xi))d\xi}d s du
%B_v(t)&=f_{0,v}(t)+\int_0^t ... B_h(t-s,...)ds
\end{flalign*}
where $f_{0,v}(t)=S_v^0\int_0^\infty \beta_h(s+t) i_{h,0}(s)e^{-\int_s^{s+t}(\mu_h(\xi)+\delta(\xi)+r_1(\xi))d\xi}ds \underset{t\to \infty}{\to} 0$. Now define the vector
$$V(t)=\begin{pmatrix}
B_h(t,\cdot) \\
B_v(t)
\end{pmatrix}\in L^1(0,\infty)\times \R$$
then we see that $V$ satisfies the following equation:
$$V(t)=f_0(t)+\int_0^t G(s)(V(t-s))ds$$
where $f_0(t)=(f_{0,h}(t,\cdot),f_{0,v}(t))^T\in L^1(0,\infty)\times \R$ encounters for the initial data and $G\in \mathcal{L}(L^1(0,\infty)\times \R)$ is the linear operator defined by:
$$G(s)\begin{pmatrix}
B_h \\
B_v
\end{pmatrix}=\begin{pmatrix}
\beta_v s^0_h e^{-\mu_v s}B_v\\
S_v^0 \int_0^\infty \beta_h(s+\xi)B_h(\xi)e^{-\int_\xi^{\xi+s}(\mu_h(u)+\delta(u)+r_1(u))du}d\xi
\end{pmatrix}$$
for each $s\geq 0$. Note that this operator $G(s)$ is called the \textit{net reproduction operator} \cite{Inaba2017} and maps the density of newborns to the density of their children produced at time $s$ later. It follows \cite{diekmann,Inaba2017} that the next generator operator is $K:L^1(0,\infty) \times \mathbb{R} \rightarrow L^1(0,\infty) \times \mathbb{R}$ and is defined by 
$$K\begin{pmatrix}
B_1 \\
B_2
\end{pmatrix}=\int_0^\infty G(s)\begin{pmatrix}
B_1 \\
B_2
\end{pmatrix}ds
=\begin{pmatrix}
\dfrac{\beta_v s^0_h B_2}{\mu_v}\\
S^0_v \int_0^\infty \int_0^\infty \beta_h(\xi+s)B_1(\xi)e^{-\int_\xi^{\xi+s}(\mu_h(u)+\delta(u)+r_1(u))du}d\xi ds
\end{pmatrix}=\begin{pmatrix}
K_1(B_1, B_2)^T \\
K_2(B_1, B_2)^T 
\end{pmatrix}.$$
From this operator $K$, we deduce that the basic reproduction number $\RR _0$ is defined by the spectral radius of $K$ that is $\RR _0=r_\sigma(K)$. We first prove that $K$ is a compact operator. Since $K_2$ has finite dimensional range and is clearly bounded, it is then sufficient to prove the compactness of $K_1$. Let $h\in \R_+$ and $S\subset L^1(0,\infty)\times \R$ be a bounded subset, so there exists $M>0$ such that $\|B_1\|_{L^1}+|B_2|\leq M$ for each $(B_1,B_2)\in L^1(0,\infty)\times \R$. We denote by $\tau_h$ the translation operator in $L^1$, \textit{i.e.}:
$$\tau_h(B_1)=B_1(\cdot+h) \qquad \forall B_1 \in L^1(0,\infty).$$ 
Let $(B_1, B_2)\in L^1(0,\infty)\times \R$, then 
$$\|\tau_h K_1(B_1,B_2)-K_1(B_1,B_2)\|_{L^1(0,\infty)}\leq \dfrac{M}{\mu_v}\|\tau_h(\beta_v s^0_h)-\beta_v s^0_h\|_{L^1(0,\infty)}\underset{h\to 0}{\to} 0$$
uniformly since $\beta_v s^0_h\in L^1(0,\infty)$. We also prove that
$$\sup_{(B_1,B_2)\in S}\int_r^\infty |K_1(B_1,B_2)(a)|da\leq \dfrac{M}{\mu_v \mu_0}\|\beta_v\|_{L^\infty}e^{-\mu_0 r}\underset{r\to \infty}{\to} 0.$$
It follows by the Riesz-Fréchet-Kolmogorov (RFK) criterion (see \textit{e.g.} \cite[Theorem X.1, p. 275]{yosida}) that $K_1$ and $K$ are compact operators, so the spectrum of $K$ is only composed of eigenvalues with finite algebraic multiplicity. Now, computing $K^2$, we see that this operator can be rewritten as
$$K^2\begin{pmatrix}
B_1 \\
B_2
\end{pmatrix}=\begin{pmatrix}
F_1(B_1) \\
F_2(B_2)
\end{pmatrix}$$
for each $(B_1,B_2)\in L^1(0,\infty)\times \R$, with $F_1, F_2$ the linear operators respectively defined on $L^1(0,\infty)$ and $\R$ by:
$$F_1(B_1)=\beta_v s^0_h \dfrac{S^0_v}{\mu_v}\int_0^\infty \int_0^\infty \beta_h(\xi+s)B_1(\xi)e^{-\int_\xi^{\xi+s}(\mu_h(u)+r_1(u)+\delta(u))du}d\xi ds$$
and
$$F_2(B_2)=\dfrac{S^0_v B_2}{\mu_v}\int_0^\infty \int_0^\infty \beta_h(\xi+s)\beta_v(\xi)s^0_h(\xi)e^{-\int_\xi^{\xi+s}(\mu_h(u)+r_1(u)+\delta(u))du}d\xi ds.$$
Let us set $\lambda_0=\RR_0^2$. We notice that $\lambda_0$ is the only eigenvalue of $F_2$, whence $r_\sigma(F_2)=\lambda_0$. We also see that
$$r_\sigma(K)=\sqrt{r_\sigma(K^2)}=\sqrt{\max\{r_\sigma(F_1),r_\sigma(F_2)\}}=\sqrt{\max\{r_\sigma(F_1),\lambda_0\}}.$$
We now shall prove that $r_\sigma(F_1)=\lambda_0$. For that, we define the set
$$\Omega:=\left\{\xi\geq 0: \beta_v(\xi)\int_0^\infty \beta_h(s+\xi)e^{-\mu_0(\xi+s)}ds>0\right\}$$
and the restriction of $F_1$ to $L^1(\Omega)$ denoted by $\tilde{F}_1\in \L(L^1(\Omega))$:
$$\tilde{F}_1(\tilde{B})(a)=\mathbf{1}_{\Omega}(a)F_1(B)(a)$$
for each $a\in \Omega$, $B\in L^1(\Omega)$ and with 
$$B(a)=\begin{cases}
\tilde{B}(s) & \text{a.e. } \quad a\in \Omega \\
0 & \text{else.}
\end{cases}$$
From Assumption \ref{Assumption:Omega} it is clear that $\tilde{F}_1$ is a positive and compact operator on $L^1(\Omega)$. Moreover, since $\Omega\neq \emptyset$, it follows that $\tilde{F}_1$ is irreducible, \textit{i.e.} 
$$\tilde{F}_1(B)(a)>0 \quad \text{a.e.} \quad a\in \Omega, \ \forall B\in L^1_+(\Omega)\setminus\{0\}$$
that is, sends the positive cone of $L^1_+(\Omega)$ on the subset of $L^1_+(\Omega)$ of functions almost everywhere strictly positive. We now observe that
$$F_1(\beta_v s^0_h)=\lambda_0 \beta_v s^0_h$$
so that $\lambda_0$ is an eigenvalue of $F_1$ associated to the eigenfunction $\beta_v s^0_h \in L^1(\Omega)$ that is positive a.e. on $\Omega$. It comes that $r_\sigma(F_1)\geq \lambda_0>0$. It follows from \cite[Lemma 4.2.10, p. 269]{MeyerNieberg91} that there exists a positive eigenfunction $\phi\in L^1_+(0,\infty)\setminus\{0\}$ such that 
$$F_1(\phi)=r_\sigma(F_1)\phi.$$
We remark that
$$\tilde{F}_1(\mathbf{1}_{\Omega}\phi)=r_\sigma(F_1)\mathbf{1}_{\Omega}\phi$$
so that $r_\sigma(F_1)$ is also an eigenvalue of $\tilde{F}_1$ associated to $\mathbf{1}_{\Omega}\phi\in L^1_+(\Omega)\setminus\{0\}$. Using a version of the Krein-Rutman theorem on the linear operator $\tilde{F}_1$ (see \cite[Corollary 4.2.15, p. 273]{MeyerNieberg91}), it follows that the spectral radius of $\tilde{F}_1$ is the only eigenvalue associated to a positive eigenfunction, whence $r_\sigma(F_1)=r_\sigma(\tilde{F}_1)$. By noticing on the other hand that
$$\tilde{F}_1(\mathbf{1}_{\Omega} \beta_v s^0_h)=\lambda_0 \mathbf{1}_{\Omega} \beta_v s^0_h$$
we deduce that $\lambda_0$ is an eigenvalue of $\tilde{F}_1$ associated to $\mathbf{1}_{\Omega} \beta_v s^0_h \in L^1_+(\Omega)$. Again, by Krein-Rutman theorem, we deduce that $r_\sigma(\tilde{F}_1)=\lambda_0$. We can then conclude that 
$$r_\sigma(K)=\sqrt{\lambda_0}=\RR_0.$$
\end{proof}

\subsection{Local stability of the parasite-free equilibrium}

In this section, we handle the local stability of the parasite-free equilibrium. Let $E=(s_h^*, i_h^*, r_h^*, S_v^*, I_v^*)\in X_+$ be an equilibrium, then we define $DF_{E}:X_0\to X$ the differential operator of $F$ around $E$ by:
\begin{equation}
\begin{array}{rcl}
DF_{E}:\begin{pmatrix}
0 \\
s_h \\
i_h \\
r_h \\
S_v \\
I_v
\end{pmatrix}&=&(DF_E)_1 \begin{pmatrix}
0 \\
s_h \\
i_h \\
r_h \\
S_v \\
I_v
\end{pmatrix}+(DF_E)_2 \begin{pmatrix}
0 \\
s_h \\
i_h \\
r_h \\
S_v \\
I_v
\end{pmatrix}+(DF_E)_3 \begin{pmatrix}
0 \\
s_h \\
i_h \\
r_h \\
S_v \\
I_v
\end{pmatrix}\\
&=&\begin{pmatrix}
0 \\
-\beta_v s_h^* I_v\\
\beta_v s_h^* I_v\\
0\\
-S^*_v\int_0^\infty \beta_h(a)i_h(a)da-S_v\int_0^\infty \beta_h(a)i_h^*(a)da \\
S^*_v\int_0^\infty \beta_h(a)i_h(a)da+S_v\int_0^\infty \beta_h(a)i_h^*(a)da
\end{pmatrix}+\begin{pmatrix}
0 \\
-\beta_v s_h I_v^*\\
0 \\
0 \\
0 \\
0
\end{pmatrix}+\begin{pmatrix}
0 \\
r_2 r_h \\
\beta_v s_h I_v^* \\
r_1 i_h \\
0 \\
0
\end{pmatrix}.
\end{array}
\end{equation}
We now state the following lemma, which brings the stability analysis to the search of eigenvalues. We first define $(A+DF_{E})_0$ the part of $A+DF_{E}$ in $\overline{D(A)}$ and $\{T_{(A+DF_{E})_0}(t)\}_{t\geq 0}$ the $C_0$-semigroup generated by $(A+DF_{E})_0$.

\begin{lemma}\label{Lemma:spectra}
The subset 
$$\{\lambda\in \sigma((A+DF_{E})_0), \Re(\lambda)\geq -\mu_0\}$$
is finite and composed at most of isolated eigenvalues with finite algebraic multiplicity, where $\sigma$ denotes the spectrum of the corresponding operator.
\end{lemma}
\begin{proof}
The linear operator $(DF_E)_1:X_0\to X$ is clearly compact so we get
$$\omega_{\ess}(\{T_{(A+DF_E)_0}(t)\}_{t\geq 0})\leq\omega_{\ess}(\{T_{(A+(DF_E)_2+(DF_E)_3)_0}(t)\}_{t\geq 0})$$
by \cite[Theorem 1.2]{DucrotMagal2008} where $\omega_{\ess}$ denotes the essential growth bound and $(A+(DF_E)_2+(DF_E)_3)_0$ is the part of $A+(DF_E)_2+(DF_E)_3$ in $\overline{D(A)}$, \textit{i.e.} $(A+(DF_E)_2+(DF_E)_3)_0:\overline{D(A)}\ni x\longmapsto (A+(DF_E)_2+(DF_E)_3)x\in \overline{D(A)}$. Now, by definition, we know that
$$\omega_{\ess}(\{T_{(A+(DF_E)_2+(DF_E)_3)_0}(t)\}_{t\geq 0})\leq \omega_{0}(\{T_{(A+(DF_E)_2+(DF_E)_3)_0}(t)\}_{t\geq 0})$$
where $\omega_0$ denotes the growth bound. To compute this quantity, we first check that the semigroup \break $\{T_{(A+(DF_E)_2+(DF_E)_3)_0}(t)\}_{t\geq 0}$ is positive. To this end, as in the proof of Proposition \ref{theorem_hille}, let $(\kappa, \phi_1,\phi_2,\phi_3,y_1,y_2) \in X$ and $(0, \psi_1,\psi_2,\psi_3,x_1,x_2) \in D(A)$ such that
$$(\lambda I - (A+(DF_E)_2))(0,\psi_1,\psi_2,\psi_3,x_1,x_2) = \break(\kappa, \phi_1,\phi_2,\phi_3,y_1,y_2).$$
We have
\[
\left\{ \begin{array}{ll}
\psi_1(a) & = \kappa e^{-\int_0^a (\mu_h(s)+\beta_v(s)I_v^*+\lambda)ds} +  \int_0^a \phi_1(s) e^{-\int_s^a (\mu_h(\tau)+\beta_v(s)I_v^*+\lambda)d \tau} ds \\
\psi_2(a) & = \int_0^a \phi_2(s) e^{-\int_s^a (\mu_h(\tau)+\delta(\tau)+r_1(\tau)+\lambda)d \tau} ds \\
\psi_3(a) & = \int_0^a \phi_3(s) e^{-\int_s^a (\mu_h(\tau)+r_2(\tau)+\lambda)d \tau} ds \\
 x_1 & = \frac{1}{\lambda + \mu_v} y_1 \\
x_2 & = \frac{1}{\lambda + \mu_v}y_2
\end{array}  
\right.\]
for each $a \in (0,+\infty)$. This proves that the linear operator $A+(DF_E)_2$ is resolvent positive and the semigroup $\{T_{(A+(DF_E)_2)_0}(t)\}_{t\geq 0}$ is positive. Since the linear operator $(DF_E)_3\in \mathcal{L}(X_0,X)$ is positive then by using \cite[Theorem 1.1]{voigt89}, one has
$$(\lambda-(A+(DF_E)_2+(DF_E)_3))^{-1}=(\lambda-(A+(DF_E)_2)))^{-1}\sum_{n=0}^\infty \left((DF_E)_3(\lambda-(A+(DF_E)_2))^{-1}\right)^n$$
which implies that $A+(DF_E)_2+(DF_E)_3$ is resolvent positive and then the semigroup $\{T_{(A+(DF_E)_2+(DF_E)_3)_0}(t)\}_{t\geq 0}$ generated by $(A+(DF_E)_2+(DF_E)_3)_0$ is positive. Since the semigroup is positive, we can now compute the operator norm of $\{T_{(A+(DF_E)_2+(DF_E)_3)_0}(t)\}_{t\geq 0}$ as
$$\|T_{(A+(DF_E)_2+(DF_E)_3)_0}(t)\|_{\mathcal{L}(X_0)}=\sup_{x\in X_{0+}, \|x\|_{X_0}=1}\|T_{(A+(DF_E)_2+(DF_E)_3)_0}(t)x\|_{X_0}.$$
Let $x=(0,s_{h,0},i_{h,0},r_{h,0},S_{v,0},I_{v,0})\in X_0$ with $\|x\|_{X_0}=1$. Following the proof of Theorem \ref{theorem_global}, we denote by
$$T_{(A+(DF_E)_2+(DF_E)_3)_0}(t)x=(0,s_h(t,.),i_h(t,.),r_h(t,.),S_v(t),I_v(t))$$
the positive solution (which exists and is unique since $A+(DF_E)_2+(DF_E)_3$ is a Hille-Yosida operator). We see that
$$s_h(t,a)=\begin{cases} 
s_{h,0}(a-t)e^{-\int_{a-t}^a \mu_h(s)ds}+\int_{a-t}^a (r_2(s)r_h(t-a+s,s)-\beta_v(t-a+s)I_v^*s_h(t-a+s))e^{-\int_s^a \mu_h(\xi)d\xi}ds & \text{ if } a\geq t,\\
\int_0^a (r_2(s)r_h(t-a+s,s)-\beta_v(t-a+s)I_v^*s_h(t-a+s))e^{-\int_s^a \mu_h(\xi)d\xi}ds & \text{ else.}
\end{cases}$$
$$i_h(t,a)=\begin{cases} 
i_{h,0}(a-t)e^{-\int_{a-t}^a \mu_h(s)ds}+\int_{a-t}^a (\beta_v(t-a+s)I_v^*s_h(t-a+s)-(r_1(s)+\delta(s))i_h(t-a+s,s))e^{-\int_s^a \mu_h(\xi)d\xi}ds & \text{ if } a\geq t,\\
\int_0^a (\beta_v(t-a+s)I_v^*s_h(t-a+s)-(r_1(s)+\delta(s))i_h(t-a+s,s))e^{-\int_s^a \mu_h(\xi)d\xi}ds & \text{ else.}
\end{cases}$$
$$r_h(t,a)=\begin{cases} 
r_{h,0}(a-t)e^{-\int_{a-t}^a \mu_h(s)ds}+\int_{a-t}^a (r_1(s)i_h(t-a+s,s)-r_2(s)r_h(t-a+s,s))e^{-\int_s^a \mu_h(\xi)d\xi}ds & \text{ if } a\geq t,\\
\int_0^a (r_1(s)i_h(t-a+s,s)-r_2(s)r_h(t-a+s,s))e^{-\int_s^a \mu_h(\xi)d\xi}ds & \text{ else.}
\end{cases}$$
Summing the three equations leads to
$$n_h(t,a)\leq \begin{cases} n_{h,0}(a-t)e^{-\int_{a-t}^a \mu_h(s)ds}\leq n_{h,0}(a-t)e^{-\mu_0 t} & \text{ if } a\geq t,\\
0 & \text{ else.}
\end{cases}$$
Since we have
$$S_v(t)=S_{v,0}e^{-\mu_v t} \qquad I_v(t)=I_{v,0}e^{-\mu_v t}$$
it follows that
$$\|T_{(A+(DF_E)_2+(DF_E)_3)_0}(t)x\|_{X_0}\leq e^{-\mu_0t}\|x\|_{X_0}$$
whence 
$$\omega_{0}(\{T_{(A+(DF_E)_2+(DF_E)_3)_0}(t)\}_{t\geq 0})\leq -\mu_0<0.$$
Consequently, the essential growth bound is negative and the statement is true by means of \cite[Cor. IV. 2.11 p. 258]{EngelNagel2000}.
\end{proof}

\begin{proposition}\label{Prop:LAS_E0}
If $\RR_0<1$ then $E_0$ is locally asymptotically stable; if $\RR_0>1$ then $E_0$ is unstable.
\end{proposition}
\begin{proof} Using \cite[Proposition 5.7.3, p. 246]{magal2018}, we know that if $s(A+(DF_{E_0})_0)<0$ then $E_0$ is locally asymptotically stable, and similarly if $s(A+(DF_{E_0})_0)>0$ then $E_0$ is unstable 
by using \cite[Proposition 5.7.4, p. 247]{magal2018} , where $s(\cdot)$ refers to the spectral bound. Moreover, by Lemma \ref{Lemma:spectra}, we only need to study the punctual spectrum of $(A+DF_{E_0})_0$. We then consider exponential solutions, \textit{i.e.} of the form $u(t)=v e^{\lambda t}$, with $0\neq v:=(s_h, i_h, r_h, S_v, I_v)\in \X$ and $\lambda\in \C$. We obtain the following system:
\begin{equation*}
\left\{
\begin{array}{rcl}
s_h'(a)&=&-\beta_v(a)I_v s_h^0(a)-(\mu_h(a)+\lambda)s_h(a)+r_2(a)r_h(a) \\
i_h'(a)&=&\beta_v(a)I_v s_h^0(a)-(\mu_h(a)+\lambda+\delta(a)+r_1(a))i_h(a) \\
r_h'(a)&=&r_1 i_h(a)-(r_2(a)+\lambda+\mu_h(a))r_h(a) \\
0&=& S_v^0\int_0^\infty \beta_h(a)i_h(a)da+\mu_v S_v \\
0&=&(\mu_v+\lambda)(S_v+I_v)
\end{array}
\right.
\end{equation*}
leading to either $I_v=-S_v$ or $\lambda=-\mu_v<0$. Supposing now that $I_v=-S_v$, we get:
$$I_v=\dfrac{S^0_v \int_0^\infty \beta_h(a)i_h(a)da}{\mu_v}$$
and
$$i_h(a)=\left(\dfrac{S^0_v \int_0^\infty \beta_h(a)i_h(a)da}{\mu_v}\right)\int_0^a \beta_v(s)s^0_h(s)e^{-\int_s^a (\mu_h(z)+\delta(z)+r_1(z)+\lambda)dz}ds$$
for each $a\geq 0$. We then obtain:
$$1=\dfrac{S^0_v}{\mu_v}\int_0^\infty \beta_h(a)
\int_0^a \beta_v(s)s^0_h(s)e^{-\int_s^a (\mu_h(z)+\delta(z)+r_1(z)+\lambda)dz}ds da=:g(\lambda).$$
We see that $g(0)=\RR_0$ so, if $\RR_0>1$ then, by continuity, one may find $\lambda>0$ such that $g(\lambda)=1$ and $E_0$ is unstable. Now, let us suppose that $\RR_0<1$, then considering $\lambda=\Re(\lambda)+i  \Im(\lambda)$, we get
$$1=|g(\lambda)|\leq \dfrac{S^0_v}{\mu_v}\int_0^\infty \beta_h(a)
\int_0^a \beta_v(s)s^0_h(s)e^{-\int_s^a (\mu_h(z)+\delta(z)+r_1(z)+\Re(\lambda))dz}ds da\leq \RR_0<1$$
whenever $\Re(\lambda)\geq 0$. As a result, $E_0$ is locally asymptotically stable when $\RR_0<1$.
\end{proof}

\subsection{Attractiveness of the parasite-free equilibrium}

We start this section by reminding classical definitions about orbits and limit sets (see \cite{Hale88}). In the sequel, for $\widehat{x}\in \X$, we will denote by $\gamma^+(\widehat{x})=\{\Phi_t(\widehat{x}), \ t\geq 0\}$ the positive orbit starting from $\widehat{x}$ and
$$\omega(\widehat{x})=\cap_{\tau \geq 0}\overline{\{\Phi_t(\widehat{x}), \ t\geq \tau\}}$$
the $\omega$-limit set of $\widehat{x}$. Also, let a function $\phi:(-\infty,0]\to \X$ such that $\phi(0)=\widehat{x}$ and for any $s\leq 0$, $\Phi_t(\phi(s))=\phi(t+s)$ for each $0\leq t\leq -s$, then the set $\{\phi(s), s\leq 0\}$ is called a negative orbit through $\widehat{x}$, denoted by $\gamma^-(\widehat{x})$. Similarly, for a function $\phi:\R\to \X$ such that $\phi(0)=\widehat{x}$ and for each $s\in \R$, $\Phi_t(\phi(s))=\phi(t+s)$ for every $t\geq 0$, then the set $\{\phi(s), s\in \R\}$ is called a complete orbit through $\widehat{x}$, denoted by $\gamma(\widehat{x})$. Let
$$H(t,\widehat{x})=\{\widehat{y}\in \X: \text{ there is a negative orbit through } \widehat{x} \text{ defined by } \phi:(-\infty,0]\to \X \text{ with } \phi(0)=\widehat{x} \text{ and } \phi(-t)=\widehat{y}\}$$
then we define the $\alpha$-limit set of $\widehat{x}$ as
$$\alpha(\widehat{x})=\cap_{\tau \geq 0}\overline{\{H(t,\widehat{x}), \ t\geq \tau\}}.$$
For a given complete orbit $\gamma(z)=\{\phi(s),s\in \R\}$, we define the $\alpha$-limit set of the orbit similarly as
$$\alpha(\phi)=\alpha(\gamma(z))=\bigcap_{\tau\geq 0}\overline{\{\phi(-t), t\geq \tau\}}.$$
We finally define the set
$$\Gamma^-(\widehat{x})=\bigcup_{t\geq 0}H(t,\widehat{x})$$
that contains all negative orbits through $x_0$. Following \cite{perasso19}, we prove now the existence of a compact attractor by proving the relative compactness of the positive orbits. 

\begin{lemma}\label{Lemma:orbit_compact}
For every $\widehat{x}\in \X_+$, the positive orbit $\gamma^+(\widehat{x})\subset \X_+$ is relatively compact, \textit{i.e.} $\overline{\gamma^+(\widehat{x})}$ is compact.
\end{lemma}

\begin{proof}
Firstly, by means of \eqref{Eq:Semiflow_h2} and estimates \eqref{Eq:Estimates_sol1}-\eqref{Eq:Estimates_sol2}, we see that for each $r>0$ and each $\widehat{x}\in B_{\X}(0,r)\cap \X_+$, we have:
$$\|\Phi_t^{h,1}(\widehat{x})\|_{(L^1(\R_+))^3}\leq r e^{-\mu_0 t}+\left(\|r_2\|_{L^\infty}+\|\beta_v\|_{L^\infty}\frac{\Lambda_v}{\mu_0}+\|r_1\|_{L^\infty}\right)r t e^{-\mu_0 t}$$
where $B_{\X}(0,r)$ in defined by \eqref{Eq:ball_r}. Secondly, due to Theorem \ref{theorem_global}, we see that $\Phi_t^v$ maps bounded sets of $\X_+$ into sets with compact closure in $\X_+$ for every $t>0$. Thus, we only need to prove the compactness for $\Phi_t^{h,2}$ for some $t>0$.

Let $t>0$ and let $S\subset \X_+$ be a bounded subset with $r=\sup_{\widehat{x}\in S}\|\widehat{x}\|_{\X}$. The goal is to show the relative compactness of $\{\Phi_t^{h,2}(\widehat{x}), \widehat{x}\in S\}\subset (L^1_+(\R_+))^3$. Defining $C_{r}=r+\frac{\Lambda_v}{\mu_v}+\frac{\Lambda_h}{\mu_0}$, we readily see from \eqref{Eq:Estimates_sol2}-\eqref{Eq:Estimates_sol3} that
\begin{equation}
\label{Eq:Phiv_estim}
\Phi_s^{S_v}(\widehat{x})\leq C_r, \quad \Phi_s^{I_v}(\widehat{x})\leq C_r, \quad \|\Phi_{s}^{s_h}(\widehat{x})\|_{L^1}\leq C_r \quad \|\Phi_{s}^{i_h}(\widehat{x})\|_{L^1}\leq C_r \quad \|\Phi_{s}^{r_h}(\widehat{x})\|_{L^1}\leq C_r \quad \forall (s,\widehat{x})\in \R_+\times S
\end{equation}
and from \eqref{Eq:Estimates_sol1} we arrive at
$$\sup_{\widehat{x}\in S}\|\Phi_t^{h,2}(\widehat{x})\|_{(L^1(\R_+))^3}\leq \dfrac{\Lambda_h}{\mu_0}+\dfrac{\Lambda_h}{\mu_0^2}\left(\|r_2\|_{L^\infty}+\|\beta_v\|_{L^\infty}C_r+\|r_1\|_{L^\infty}\right)<+\infty.$$
Let $\xi>0$. We see that:
$$\sup_{\widehat{x}\in S}\sum_{j\in\{s,i,r\}}\int_{\xi}^\infty |\Phi_t^{j_h,2}(\widehat{x})(a)|da\leq \Lambda_h \int_{\xi}^\infty e^{-\mu_0 a}da+\Lambda_h \left(\|r_2\|_{L^\infty}+\|\beta_v\|_{L^\infty}C_r+\|r_1\|_{L^\infty}\right)\int_{\xi}^\infty ae^{-\mu_0 a}da\underset{\xi\to +\infty}{\longrightarrow} 0.$$
Let $s\in[0,t]$, $\xi>0$ and $\widehat{x}\in S$. Then 
\begin{flalign}
\||\Phi_{s+\xi}^{r_h,2}(\widehat{x})-\Phi_{s}^{r_h,2}(\widehat{x})\|_{L^1}&\leq \Lambda_h \xi t\|r_1\|_{L^\infty}+\|r_1\|_{L^\infty}\int_{0}^s \int_0^a |\Phi^{i_h,2}_{s+\xi-a+w}(\widehat{x})(w)-\Phi^{i_h,2}_{s-a+w}(\widehat{x})(w)|dw da \nonumber \\
&\leq \Lambda_h \xi t\|r_1\|_{L^\infty}+\|r_1\|_{L^\infty} \int_0^s \|\Phi^{i_h,2}_{u+\xi}(\widehat{x})-\Phi^{i_h,2}_{u}(\widehat{x})\|_{L^1}du. \label{Eq:Phi_rh}
\end{flalign}
From \eqref{system}-\eqref{Eq:Phiv_estim} it comes
$$\left|\dfrac{d\Phi_{s}^{I_v}
(\widehat{x})}{ds}\right|\leq \|\beta_h\|_{L^\infty}(C_r)^2+\mu_v C_r=:\tilde{C}_r, \quad \forall (s,\widehat{x})\in \R_+\times S$$
whence
\begin{equation}
\label{Eq:PhiIv_estim}
|\Phi_{s+\xi}^{I_v}(\widehat{x})-\Phi_s^{I_v}(\widehat{x})|\leq \xi \tilde{C}_r, \quad \forall (s,\widehat{x},\xi)\in\R_+\times S\times \R_+^*.
\end{equation}
Let $s\in[0,t], \xi>0$ and $\widehat{x}\in S$, then
\begin{equation}
\label{Eq:Phi_ih}
\||\Phi_{s+\xi}^{i_h,2}(\widehat{x})-\Phi_{s}^{i_h,2}(\widehat{x})\|_{L^1}\leq \Lambda_h \xi t\|\beta_v\|_{L^\infty} C_r+\xi t\|\beta_v\|_{L^\infty}\tilde{C}_r C_r+\|\beta_v\|C_r\int_0^s \|\Phi^{s_h,2}_{u+\xi}(\widehat{x})-\Phi^{s_h,2}_{u}(\widehat{x})\|_{L^1}du
\end{equation}
and using \eqref{Eq:PhiIv_estim} we see on one hand that
\begin{flalign*}
&\Lambda_h\int_0^s |e^{-\int_0^a (\mu_h(w)+\beta_v(w)\Phi^{I_v}_{s+\xi-a+w}(\widehat{x}))dw}-e^{-\int_0^a (\mu_h(w)+\beta_v(w)\Phi^{I_v}_{s-a+w}(\widehat{x}))dw}|da \\
&\leq \Lambda_h t \times \max\left\{e^{\|\beta_v\|_{L^\infty}t\xi \tilde{C}_r}-1,1-e^{-\|\beta_v\|_{L^\infty}t\xi \tilde{C}_r}\right\}=:\kappa_{1}(\xi)
\end{flalign*}
and on the other hand that
\begin{flalign*}
& \int_0^s \left|\int_0^a r_2(w)\Phi^{r_h,2}_{s+\xi-a+w}(\widehat{x})(w)\left(e^{-\int_w^a (\mu_h(l)+\beta_v(l)\Phi^{I_v}_{s+\xi-a+l}(\widehat{x})dl}dw-e^{-\int_w^a (\mu_h(l)+\beta_v(l)\Phi^{I_v}_{s-a+l}(\widehat{x})dl}\right)dw\right|da \\
&\leq \Lambda_h t^2 \|r_2\|_{L^\infty} \times \max\left\{e^{\|\beta_v\|_{L^\infty}t \xi \tilde{C}_r}-1,1-e^{-\|\beta_v\|_{L^\infty}t \xi \tilde{C}_r}\right\}=:\kappa_2(\xi)
\end{flalign*}
whence
\begin{equation}
\label{Eq:Phi_sh}
\|\Phi_{s+\xi}^{s_h,2}(\widehat{x})-\Phi_{s}^{s_h,2}(\widehat{x})\|_{L^1}\leq \Lambda_h \xi+\Lambda_h \xi t \|r_2\|_{L^\infty}+\kappa_1(\xi)+\kappa_2(\xi)+\|r_2\|_{L^\infty}\int_0^s \|\Phi^{r_h,2}_{u+\xi}(\widehat{x})-\Phi^{r_h,2}_{u}(\widehat{x})\|_{L^1}du.
\end{equation}
We deduce from \eqref{Eq:Phi_rh}-\eqref{Eq:Phi_ih}-\eqref{Eq:Phi_sh} that for every $s\in[0,t], \xi>0$, $\widehat{x}\in S$ and each $j\in\{s,i,r\}$, we have
$$\|\Phi_{s+\xi}^{j_h,2}(\widehat{x})-\Phi_{s}^{j_h,2}(\widehat{x})\|_{L^1}\leq \kappa_3(\xi)+\kappa_4\int_0^s\|\Phi_{u+\xi}^{j_h,2}(\widehat{x})-\Phi_{u}^{j_h,2}(\widehat{x})\|_{L^1}du$$
where $\kappa_3(\xi)>0$ and $\kappa_4>0$ are independent of $\widehat{x}$ and $s$, with $\kappa_3(\xi)\underset{\xi \to 0}{\longrightarrow} 0$. Using Gronwall's inequality lead to
\begin{equation}
\label{Eq:Phi_jh}
\|\Phi_{s+\xi}^{j_h,2}(\widehat{x})-\Phi_{s}^{j_h,2}(\widehat{x})\|_{L^1}\leq \kappa_3(\xi)e^{\kappa_4 s} \quad \forall (s,\xi,\widehat{x},j)\in[0,t]\times \R_+^*\times S\times \{s,i,r\}.
\end{equation}
Now, let $\xi>0$ and $\widehat{x}\in S$, then from \eqref{Eq:Phi_jh} we get
\begin{flalign*}
\|\Phi_t^{r_h,2}(\widehat{x})(\cdot+\xi)-\Phi_t^{r_h,2}(\widehat{x})\|_{L^1}&\leq 2\Lambda_h \xi t \|r_1\|_{L^\infty}+\Lambda_h t^2 \|r_1\|_{L^\infty}(1-e^{-\xi (\|r_2\|_{L^\infty}+\|\mu_h\|_{L^\infty})})+\|r_1\|_{L^\infty}\int_0^t \|\Phi^{i_h,2}_{u+\xi}(\widehat{x})-\Phi^{i_h,2}_{u}(\widehat{x})\|_{L^1}du \\
&\leq 2\Lambda_h \xi t \|r_1\|_{L^\infty}+\Lambda_h t^2 \|r_1\|_{L^\infty}(1-e^{-\xi (\|r_2\|_{L^\infty}+\|\mu_h\|_{L^\infty})})+t\|r_1\|_{L^\infty}\kappa_3(\xi)e^{\kappa_4 t}\underset{\xi\to 0}{\longrightarrow} 0
\end{flalign*}
uniformly in $\widehat{x}\in S$. We deduce that $\Phi^{r_h,2}_t$ is a compact operator by using the RFK criterion. Similar computations show that $\Phi^{i_h,2}_t$ and $\Phi^{s_h,2}_t$ are compact operators.

Finally, the relative compactness of the positive orbits follows from \cite[Proposition 3.13, p. 100]{Webb85}.
\end{proof}

The latter compactness result of the positive orbits then leads to the existence of a compact attractor in the following sense (see \textit{e.g.} \cite[Lemma 3.1.1 and 3.1.2, p. 36]{Hale88}, or \cite[Theorem 4.1, p. 167]{Walker80}). Let us remind the definition of the semi-distance \cite{Hale88} as $d(B_1,B_2)=\sup_{y_1\in B_1}\inf_{y_2\in B_2}\|y_1-y_2\|_{\X}$ for any subset $B_1, B_2 \subset \X$.

\begin{lemma}\label{Lemma:OmegaLimit}
For every $\widehat{x}\in \X_+$:
\begin{enumerate}
\item $\omega(\widehat{x})$ is non-empty, compact and connected;
\item $\omega(\widehat{x})$ is invariant under $\Phi$, \textit{i.e.} $\Phi_t(\omega(\widehat{x}))=\omega(\widehat{x}), \forall t\geq 0$;
\item $\omega(\widehat{x})$ is an attractor, \textit{i.e.} $\lim_{t\to +\infty} d(\Phi_t(\widehat{x}), \omega(\widehat{x}))=0$. 
\end{enumerate}
\end{lemma}
In order to get the attractiveness of the parasite-free equilibrium, we will make use of Lyapunov functionals, introduced in \cite{magal2010} for age-structured models (see also \cite{perasso18,Richard20} for the computations in other age-structured models and see \cite{Vargas2009} for other computations in compartmental (SIRS) epidemiological ODEs model). To this end we define the following non-negative function
$$g:\R_+^*\ni x \longmapsto x-\ln(x)-1\in \R_+$$
and the functional
$$L_0:(s_h,i_h,r_h,S_v,I_v)\longmapsto \int_0^\infty \psi(a)i_h(a)da+g\left(\dfrac{S_v}{S_v^0}\right)+\dfrac{I_v}{S_v^0}$$
where
$$\psi(a)=\int_a^\infty \beta_h(s)e^{-\int_a^s (\mu_h(z)+\delta(z)+r_1(z))dz}ds.$$
We also remind the following definition:
\begin{definition}
Let $S\subset \X$. A function $L:\X\to \R$ is called a Lyapunov function on $S$ if there hold that:
\begin{itemize}
\item $L$ is continuous on $\overline{S}$ (the closure of $S$ in $\X$);
\item the function $\R_+\ni t\longmapsto L(\Phi_t(\widehat{x}))$ is non-increasing for every $\widehat{x}\in S$.
\end{itemize}
\end{definition}
We will now show that $L_0$ is a Lyapunov functional.
\begin{proposition}\label{Prop:L0_def}
For every $\widehat{x}\in \X_+$, the function $(t,\widehat{v})\longmapsto L_0(\Phi_t(\widehat{v}))$ is well-defined on $\R_+\times \omega(\widehat{x})$.
\end{proposition}
\begin{proof}
Let $\widehat{x}\in \X_+$, $\widehat{v}\in \omega(\widehat{x})$ and $t\geq 0$. It is clear that $L_0(\Phi_t(\widehat{v}))$ is well-defined if $\Phi_t^{S_v}(\widehat{v})>0$ which is true since we have
\begin{equation}\label{eq:minoration_sv}
\liminf_{t\to \infty}\Phi_t^{S_v}(\widehat{x})\geq \dfrac{\Lambda_v}{\mu_v+\|\beta_h\|_{L^\infty}C_h}>0
\end{equation}
with $C_h$ given by \eqref{Eq:Estimates_sol2}. Finally, since $\widehat{v}\in\omega(\widehat{x})$ it follows that $\Phi_t^{S_v}(\widehat{v})\geq\frac{\Lambda_v}{\mu_v+\|\beta_h\|_{L^\infty}C_h}>0$ for each $t\geq 0$.
\end{proof}
\begin{remark}
Note that the latter proposition is stated on the $\omega$-limit set while the function $(t,\widehat{v})\longmapsto L_0(\Phi_t(\widehat{v}))$ is well-defined on $\R_+^*\times \X_+$. However, to show that this is a Lyapunov function, we will use the upper bound of the solution formulated in Theorem \ref{theorem_global} that is reached when $t\to +\infty$. This will allow to get an upper bound according to $\RR_0-1$ which is non-positive when $\RR_0\leq 1$.
\end{remark}

We start with a property of the semiflow.

\begin{lemma}\label{Lemma:state-continuous}
The semiflow $\Phi$ is state-continuous uniformly in finite time meaning that for every $(t,\widehat{x})\in \R_+\times \X_+$ and every $\ep>0$, there exists some $\delta>0$ such that
$$ \forall \widehat{y}\in\X_+, \quad  \|\widehat{x}-\widehat{y}\|_{\X}\leq \delta \Longrightarrow \|\Phi_s(\widehat{x})-\Phi_s(\widehat{y})\|_{\X}\leq \ep \quad \forall s\in[0,t].$$
\end{lemma}

\begin{proof}
Let $x\in X_{0+}$. Then $U(.)x\in \Co(\R_+,X_{0+})$ by Proposition \ref{prop_existence} and Theorem \ref{theorem_global}. Let $t\geq 0$, $\ep>0$ and define
$$m=\max\left\{\dfrac{\Lambda_h}{\mu_0}+\dfrac{\Lambda_v}{\mu_v},\|x\|_{X}+\ep\right\}, \qquad \delta=\min\left\{\dfrac{\ep e^{-t}}{c_m},\ep\right\}$$
where $c_m$ is as in Proposition \ref{prop1}. Let $y\in X_{0+}$ such that $\|x-y\|_{X}\leq \delta$. 
We have $U(.)y\in \Co(\R_+,X_{0+})$ and $F$ is Lipschitz continuous whence $F\circ U(.)x$ and $F\circ U(.)y$ belong to $L^1([0,t],X)$. Also we have $(U(s)x,U(s)y)\in (X_{0+}\cap B_{X}(0,m))^2$ for every $s\in[0,t]$ by using \eqref{Eq:Estimates_sol1}-\eqref{Eq:Estimates_sol2} and the fact that $\|y\|_{\X}\leq \delta+\|x\|_{\X}\leq m$. Since $A$ is a Hille-Yosida operator from Proposition \ref{theorem_hille}, then it generates a locally Lipschitz continuous integrated semigroup $\{S_A(t)\}_{t\geq 0}$ on $X$ \cite[Proposition 3.4.3, p.116]{magal2018}. Hence the integrated solution writes as
\begin{equation}\label{Eq:u_convol}
    U(t)x=T_{A_0}(t)x+\dfrac{d}{dt}\left(S_A \ast (F\circ U(.)x)(t)\right)
\end{equation}
where $\ast$ denotes the convolution product:
$$(S_A \ast f)(t)=\int_0^t S_A(t-s)f(s)ds$$
$A_0$ is the part of $A$ in $X_0$ defined as
$$A_0x=Ax, \ \forall x\in D(A_0):=\{x\in D(A): Ax\in X_0\}.$$
and where $\{T_{A_0}\}_{t\geq 0}$ is the $\Co_0$-semigroup on $X_0$ generated by $A_0$, which exists by using the Hille-Yosida theorem. From \eqref{Eq:u_convol} we deduce that for every $s\in[0,t)$:
$$U(s)y-U(s)x=T_{A_0}(s)(y-x)+\dfrac{d}{dt}\left(S_A \ast ((F\circ U(.)y)-(F\circ U(.)x))(s)\right).$$
From \cite[Lemma 2.2.10, p. 65]{magal2018}, we have $\rho(A_0)=\rho(A)$ whence (by Proposition \ref{theorem_hille}):
$$\|T_{A_0}(s)(y-x)\|_{X}\leq e^{-\mu_0 s}\|y-x\|_{\X}\leq \delta, \ \forall s\in[0,t].$$
Using Kellermann-Hieber theorem \cite[Theorem 3.6.2, p.133]{magal2018}, it comes that for any $\tau>0$ and any $f\in L^1((0,\tau),X)$, the maps $t\longmapsto \left(S_A \ast f\right)(t)$ are continuously differentiable and for all $t\in[0,\tau]$ then
\begin{equation}\label{Eq:convol-estim}
\dfrac{d}{dt}(S_A \ast f)(t)\leq e^{-\mu_0 t}\int_0^t e^{\mu_0 s}\|f(s)\| \mathrm{d}s.
\end{equation}
Using \eqref{Eq:convol-estim} with $f=F\circ U(.)y-F\circ U(.)x\in L^1([0,t],X)$ leads to
$$\|U(s)y-U(s)x\|_{X}\leq \delta+\int_0^s \|F(U(\xi)y)-F(U(\xi)x)\|_{X}d\xi, \ \forall s\in[0,t].$$
It follows from Proposition \ref{prop1} that
$$\|U(s)y-U(s)x\|_{X}\leq \delta+c_m\int_0^s \|U(\xi)y-U(\xi)x\|_{\X}d\xi, \quad \forall s\in[0,t]$$
which leads by means of Gronwall's inequality to
\begin{equation*}
\|U(s)y-U(s)x\|_{X}\leq c_m e^s \delta\leq \ep, \quad \forall s\in [0,t].
\end{equation*}
This ends the proof.
\end{proof}

\begin{proposition}\label{Prop:L0_lyap}
For every $\widehat{x}\in \X_+$, the function $L_0$ is a Lyapunov function on $\omega(\widehat{x})$ whenever $\RR_0\leq 1$.
\end{proposition}
\begin{proof}
Let $\widehat{x}\in \X_+$. By Proposition \ref{Prop:L0_def}, the function $(t,\widehat{v})\longmapsto L_0(\Phi_t(\widehat{v}))$ is well-defined on $\R_+\times \omega(\widehat{x})$. Let $\widehat{v}\in \omega(\widehat{x})$ and suppose that $v=(0,\widehat{v})\in D((A+F)_0)\cap X_{0+}$ where
$$D((A+F)_0)=\{x\in D(A): Ax+F(x)\in X_0\}.$$
Then, the solution of \eqref{pb_cauchy} with initial condition $v$, which we denote by $(0,s_h,i_h,r_h,S_v,I_v)$ for convenience, belongs to $\mathcal{C}^1(\R_+,X_+)\cap \mathcal{C}(\R_+,D(A))$ by using \cite[Theorem 5.6.6, p.242]{magal2018}. It follows that the functions $s_h(t,.)$, $i_h(t,.)$ and $r_h(t,.)$ are in $W^{1,1}(\R_+)$ for each $t\geq 0$. We first compute the three following quantities: 
\begin{flalign*}
\dfrac{d}{dt}\left(\int_0^\infty \psi(a)i_h(t,a)da\right)=&\int_0^\infty \psi(a)\left(-\partial_a i_h(t,a)+\beta_v(a)I_v(t)s_h(t,a)-(\mu_h(a)+\delta(a)+r_1(a))i_h(t,a)\right)da \\
=& \int_0^\infty \psi'(a)i_h(t,a)da+\int_0^\infty \left(\beta_v(a)I_v(t)s_h(t,a)-(\mu_h(a)+\delta(a)+r_1(a))i_h(t,a)\right)da \\
=&-\int_0^\infty \beta_h(a)i_h(t,a)da+\int_0^\infty \beta_v(a)\psi(a)I_v(t)s_h(t,a)da
\end{flalign*}
\begin{flalign*}
\dfrac{d}{dt}\left(\dfrac{I_v(t)}{S_v^0}\right)&=\dfrac{S_v(t)}{S_v^0}\int_0^\infty \beta_h(a)i_h(t,a)da-\dfrac{\mu_v I_v(t)}{S_v^0}.
\end{flalign*}
and
\begin{flalign*}
\dfrac{d}{dt}\left(g\left(\dfrac{S_v(t)}{S_v^0}\right)\right)&=\dfrac{1}{S_v^0}\left(1-\dfrac{S_v^0}{S_v(t)}\right)\left(\Lambda_v-S_v(t)\int_0^\infty \beta_h(a)i_h(t,a)da-\mu_v S_v(t)\right)\\
&=-\left(\dfrac{\mu_v S_v(t)}{S_v^0}\right)\left(1-\dfrac{S_v^0}{S_v(t)}\right)^2-\dfrac{S_v(t)}{S_v^0}\int_0^\infty \beta_h(a)i_h(t,a)da+\int_0^\infty \beta_h(a)i_h(t,a)da
\end{flalign*}
for every $t>0$. Summing all the terms, we can compute the derivative of $L_0$ and see that it simply becomes:
\begin{flalign}\label{Eq:Lyap_der}
\dfrac{dL_0(\Phi_t(\widehat{v}))}{dt}&=\left(\dfrac{\mu_v I_v(t)}{S^0_v}\right)\left(\dfrac{1}{\mu_v}\int_0^\infty \beta_v(a)\psi(a)s_h(t,a) S^0_v da-1\right)
-\left(\dfrac{\mu_v S_v(t)}{S_v^0}\right)\left(1-\dfrac{S_v^0}{S_v(t)}\right)^2
\end{flalign}
for each $t>0$. It follows that
\begin{flalign}\label{Eq:Lyap_integral}
L_0(\Phi_t(\widehat{v}))&=L_0(\widehat{v})+\int_0^t \left(\left(\dfrac{\mu_v \Phi_s^{I_v}(\widehat{v})}{S^0_v}\right)\left(\dfrac{1}{\mu_v}\int_0^\infty \beta_v(a)\psi(a)\Phi_s^{s_h}(\widehat{v})(a) S^0_v da-1\right)
-\left(\dfrac{\mu_v \Phi_s^{S_v}(\widehat{v})}{S_v^0}\right)\left(1-\dfrac{S_v^0}{\Phi_s^{S_v}(\widehat{v})}\right)^2
\right)ds
\end{flalign}
Now, we want to get the same equality when $v\not\in D((A+F)_0)\cap X_{0+}$. Let $t>0$ and $\ep>0$. By density of $D((A+F)_0)\cap X_{0+}$ into $X_{0+}$ (see \cite[Lemma 5.6.7, p.243]{magal2018}) and by using Lemma \ref{Lemma:state-continuous} we know that there exists $v_{\ep}:=(0,\widehat{v}_{\ep})\in D((A+F)_0)\cap X_{0+}$ such that
$$\|v_{\ep}-v\|_{X}\leq \ep, \qquad \sup_{s\in[0,t]}\|\Phi_s(\widehat{v})-\Phi_s(\widehat{v}_{\ep})\|_{\X}\leq \ep.$$
With the expression of $L_0$, one may note that
$$|L_0(\Phi_t(\widehat{v}))-L_0(\Phi_t(\widehat{v}_{\ep}))|\leq c_1(\ep), \qquad c_1(\ep)\underset{\ep\to 0}{\to} 0$$
since we have
$$\left|g\left(\dfrac{\Phi_t^{S_v}(\widehat{v})}{S_v^0}\right)-g\left(\dfrac{\Phi_t^{S_v}(\widehat{v}_{\ep})}{S_v^0}\right)\right|\leq \dfrac{\ep}{S_v^0}\left(1+\dfrac{\mu_v+\|\beta_h\|_{L^\infty}C_h}{\mu_v}\right)$$
(by combining the mean value inequality with the lower bound \eqref{eq:minoration_sv}). Now, rewriting for simplicity the right term of \eqref{Eq:Lyap_integral} as
$$L_0(\widehat{v})+\tilde{L}(\widehat{v})(t)$$
then one may show that we have
$$|L_0(\widehat{v})+\tilde{L}(\widehat{v})(t)-L_0(\widehat{v}_{\ep})-\tilde{L}(\widehat{v}_{\ep})(t)|\leq|L_0(\widehat{v})-L_0(\widehat{v}_{\ep})|+|\tilde{L}(\widehat{v})(t)-\tilde{L}(\widehat{v}_{\ep})(t)|\leq c_2(\ep), \qquad c_2(\ep)\underset{\ep \to 0}{\to} 0.$$
Since we have $L_0(\Phi_t(\widehat{v}_{\ep}))=L_0(\widehat{v}_{\ep})+\tilde{L}(\widehat{v}_{\ep})(t)$ then it comes
$$|L_0(\Phi_t(\widehat{v}))-L_0(\widehat{v})-\tilde{L}(\widehat{v})(t)|\leq |L_0(\Phi_t(\widehat{v}))-L_0(\Phi_t(\widehat{v}_ {\ep}))|+|L_0(\Phi_t(\widehat{v}_{\ep}))-L_0(\widehat{v})-\tilde{L}(\widehat{v})(t))|\leq c_1(\ep)+c_2(\ep).$$
The latter equation being true for $\ep>0$ as small as wanted, then \eqref{Eq:Lyap_integral} is true and holds for any $t>0$. It follows that the equality \eqref{Eq:Lyap_integral} also holds for every $\widehat{v}\in \omega(\widehat{x})$. By continuity of the terms under the integral, the equality \eqref{Eq:Lyap_der} also holds for every $\widehat{v}\in \omega(\widehat{x})$. Finally, using \eqref{Eq:Estimates_sol1} and letting $t\to\infty$, we see that
$$s_h(t,a)=\Phi_t^{s_h}(v)(a)\leq \Lambda_h e^{-\int_0^a\mu_h(s)ds}=s^0_h(a)$$
for each $t\geq 0$ since $\widehat{v}\in \omega(\widehat{x})$. We deduce that
\begin{equation}\label{Eq:lyap<0}
\dfrac{dL_0(\Phi_t(\widehat{v}))}{d t}\leq \left(\dfrac{\mu_v I_v(t)}{S^0_v}\right)\left(\RR_0^2-1\right)-\left(\dfrac{\mu_v S_v(t)}{S_v^0}\right)\left(1-\dfrac{S_v^0}{S_v(t)}\right)^2\leq 0
\end{equation}
since $\RR_0\leq 1$ and consequently $L_0$ is a Lyapunov function on $\omega(\widehat{x})$.
\end{proof}

Using the Lyapunov functional defined above and the Lasalle invariance principle, we can compute the basin of attraction of the parasite-free equilibrium, which is the main result of this section.

\begin{theorem}
If $\RR_0\leq 1$ then $E_0$ is globally attractive in $\X_+$.
\end{theorem}               
\begin{proof}
Suppose that $\RR_0\leq 1$ and let $\widehat{x}\in \X_+$ so that $\omega(\widehat{x})\subset \X_+$. Let $\widehat{v}\in\omega(\widehat{x})$. Since $\omega(\widehat{x})$ is invariant, it follows by definition (see \cite{Hale88}) that there exists a complete orbit $\gamma(\widehat{v}):=\{\phi(s), s\in \R\}\subset \omega(\widehat{x})$ through $\widehat{v}$. From Proposition \ref{Prop:L0_lyap}, we know that $L_0$ is a Lyapunov function on $\omega(\widehat{x})$ (and in particular on $\gamma(\widehat{v})$). Using \cite[Prop. 2.51 p. 53]{SmithThieme2011} we deduce that $L_0$ is constant on $\alpha(\phi)$ and on $\omega(z)$. Since $\gamma(\widehat{v})\subset \omega(\widehat{x})$ which is compact, then $\gamma^-(\widehat{v}):=\{\phi(s),s\leq 0\}$ is relatively compact in $\X$ and non empty. From \cite[Theorem 2.48, p.52]{SmithThieme2011} the alpha-limit set $\alpha(\phi)$ is non-empty, compact, invariant, connected and $\lim_{t\to -\infty}d(\phi(t),\alpha(\phi))=0$. We show that $\alpha(\widehat{v})$ is reduced to $E_0$. Let $\widehat{z}:=(s_h,i_h,r_h,S_v,I_v)\in \alpha(\widehat{v})$. Since $\alpha(\widehat{v})$ is invariant then there exists a complete orbit $\gamma(\widehat{z}):=\{\phi_2(s),s\in \R\}\subset \alpha(\widehat{v})$ on which $L_0$ is constant. It implies that 
$$\dfrac{d}{dt}L_0(\Phi_t(\phi_2(s))=0, \quad \forall t>0 \quad \text{and} \quad \forall  s\in \R.$$
From \eqref{Eq:lyap<0} we get $\Phi_t^{S_v}(\phi_2(s))=S_v^0$ for every $s\in \R$ and every $t>0$ (whence for every $t\geq 0$ since $\Phi_0^{S_v}(\phi_2(s))=\Phi_{t}^{S_v}(\phi_2(s-t))=S_v^0$ for every $t>0$). In particular, taking $s=-t$ leads to $S_v=\Phi_t^{S_v}(\phi_2(-t))=S_v^0$. Also, since $\alpha(\widehat{v})\subset \omega(\widehat{x})$ then from \eqref{Eq:Estimates_sol3} we deduce that $\Phi_t^{S_v}(\phi_2(s))+\Phi_t^{I_v}(\phi_2(s))=\frac{\Lambda_v}{\mu_v}$ for every $(t,s)\in \R_+\times\R$, hence $\Phi_t^{I_v}(\phi_2(s))=0$ for every $(t,s)\in \R_+\times\R$ (and in particular $I_v=0$). It follows, by using \eqref{Eq:Semiflow_h2}, that $\Phi_t^{i_h}(\phi_2(s),a)=0$ for every $(t,s,a)\in\R_+\times\R\times [0,t]$. Actually we even get
$$\Phi_t^{i_h}(\phi_2(s),a)=0, \quad \forall (t,s,a)\in \R_+\times\R\times \R_+$$
since for every $a\geq t$ we remark that
$$\Phi_t^{i_h}(\phi_2(s),a)=\Phi_{t+a}^{i_h}(\phi_2(s-a),a)=0, \quad \forall (t,s,a)\in \R_+\times \R\times [t,+\infty)$$
whence $i_h\equiv 0$. Again with \eqref{Eq:Semiflow_h2} we get
$$\Phi_t^{r_h}(\phi_2(s),a)=0, \quad \text{and} \quad \Phi_t^{s_h}(\phi_2(s),a)=s_h^0(a), \quad \forall (t,s,a)\in \R_+\times\R\times \R_+$$
hence $\widehat{z}=E_0$ and $\alpha(\widehat{v})=\{E_0\}$, \textit{i.e.} $\lim_{t\to -\infty}d(\phi(t),\{E_0\})=0$. Since $L_0(E_0)=0$ then we have $L_0(\phi(s))=0$ for every $s\in \R$ since $L_0$ is a Lyapunov function on $\gamma(\widehat{v})$. This implies that $L_0$ is constant on $\gamma(\widehat{v})$. Following the above arguments, we show that $\widehat{v}=E_0$. In conclusion $\omega(\widehat{x})=\{E_0\}$ for each $\widehat{x}\in \X_+$ and $E_0$ is globally attractive in $\X_+$.
\end{proof}

\subsection{Global stability of the parasite-free equilibrium}

In this section, we handle the stability of $E_0$ in the case $\RR_0=1$ that is when the principle of linearisation fails (see Proposition \ref{Prop:LAS_E0}). In \cite{Richard20}, the stability was proved by making use of the Lyapunov function and the idea that taking an initial condition close to the equilibrium is equivalent to have a Lyapunov function small enough at this point, thus controlling some energy at each time. In our case, it cannot directly be applied since we have a Lyapunov function on $\omega(z)$ for each $z\in \X_+$, so it is difficult to control the energy through time. However, using estimates on $s_h$, on $S_v$ and using the Lyapunov function $L_0$ we can then deal with the stability by showing directly the definition.

\begin{theorem}\label{Thm:stab_R0=1}
Suppose that $\RR_0=1$, then the disease-free equilibrium $E_0$ is Lyapunov stable.
\end{theorem}
\begin{proof}Let $\ep>0$, $\eta>0$ and $\widehat{z}:=(s_h^\eta, i_h^\eta, r_h^\eta, S_v^\eta, I_v^\eta)\in B_{\X}(E_0,\eta)$, implying that $\|E_0-\widehat{z}\|_{\X}\leq \eta$. Without loss of generality, we suppose that $\eta\leq 1$. We still denote 
$$(s_h, i_h, r_h, S_v, I_v)(t)=\Phi_t(\widehat{z})$$
the solution of \eqref{system} at time $t\geq 0$ with the initial condition $\widehat{z}$. The goal is to prove that for $\eta$ small enough, then we have:
$$\|\Phi_t(\widehat{z})-E_0\|_{\X}\leq \ep$$
for each $t\geq 0$. Using \eqref{Eq:Estimates_sol1} we see that the following estimates hold:
\begin{equation}\label{Eq:Sh_estimate}
s_h(t,a)\leq \begin{cases}
s_h^0(a) & \text{if } t> a \\
(s_h^\eta+i_h^\eta+r_h^\eta)(a-t)e^{-\int_{a-t}^a \mu_h(s)ds} & \text{if } a\geq t
\end{cases}
\end{equation}
\begin{equation*}
i_h(t,a)\leq \begin{cases}
s_h^0(a) & \text{if } t> a \\
(s_h^\eta+i_h^\eta+r_h^\eta)(a-t)e^{-\int_{a-t}^a \mu_h(s)ds} & \text{if } a\geq t
\end{cases}
\end{equation*}
and using \eqref{Eq:Estimates_sol3} we get
$$S_v(t)\leq S_v^0+2\eta e^{-\mu_v t}, \qquad I_v(t)\leq S_v^0+2\eta e^{-\mu_v t}$$
for each $t\geq 0$. Reminding \eqref{Eq:Lyap_der} and using \eqref{Eq:Sh_estimate}, we see that:
\begin{flalign*}
\dfrac{dL_0(\Phi_t(v))}{d t}&=\left(\dfrac{\mu_v I_v(t)}{S^0_v}\right)\left(\dfrac{1}{\mu_v}\int_0^\infty \beta_v(a)\psi(a)s_h(t,a) S^0_v da-1\right)
-\left(\dfrac{\mu_v S_v(t)}{S_v^0}\right)\left(1-\dfrac{S_v^0}{S_v(t)}\right)^2 \\
&\leq I_v(t)\|\beta_v \psi\|_{L^\infty}\int_t^\infty (s_h^\eta-s_h^0+i_h^\eta+r_h^\eta)(a-t)e^{-\int_{a-t}^a \mu_h(s)ds}da \\
&\leq  \left(S_v^0+2\eta e^{-\mu_v t}\right)3 \eta \|\beta_v \psi\|_{L^\infty}e^{-\mu_0 t}\\
&\leq C\eta e^{-\mu_0 t}
\end{flalign*}
for each $t\geq 0$ where $C$ is some constant given by the parameters and that is independent of $\eta$. It follows that
\begin{flalign*}
L_0(\Phi_t(\widehat{z}))&\leq L_0(\Phi_0(\widehat{z})))+\int_0^t C\eta e^{-\mu_0 s}ds \\
&\leq \int_0^\infty \psi(a)i_h^\eta(a)da+g\left(\dfrac{S_v^\eta}{S_v^0}\right)+\dfrac{I_v^\eta}{S_v^0}+\dfrac{C \eta}{\mu_0}(1-e^{-\mu_0 t})\\
&\leq \eta\|\psi\|_{L^\infty}+\dfrac{S_v^\eta}{S_v^0}-\ln\left(\dfrac{S_v^\eta}{S_v^0}\right)-1+\dfrac{\eta}{S^0_v}+\dfrac{C \eta}{\mu_0}(1-e^{-\mu_0 t}) \\
&\leq \eta\left(\|\psi\|_{L^\infty}+\dfrac{1}{S_v^0}+\dfrac{C}{\mu_0}\right)+\dfrac{S_v^\eta}{S_v^0}+\ln\left(\dfrac{S_v^0}{S_v^\eta}\right)-1\\
&\leq \eta\left(\|\psi\|_{L^\infty}+\dfrac{1}{S_v^0}+\dfrac{C}{\mu_0}+\dfrac{1}{S_v^0}\right)+\dfrac{S_v^0}{S_v^\eta}-1\\
&\leq \eta\left(\|\psi\|_{L^\infty}+\dfrac{1}{S_v^0}+\dfrac{C}{\mu_0}+\dfrac{1}{S_v^0}\right)+\dfrac{\eta}{S_v^0-\eta}\\
\end{flalign*}
where we used the fact that $\ln(x)\leq x-1$. Still without loss of generality, we can suppose that $\eta\leq \frac{S_v^0}{2}$ so that
$$L_0(\Phi_t(\widehat{z}))\leq \eta\left(\|\psi\|_{L^\infty}+\dfrac{1}{S_v^0}+\dfrac{C}{\mu_0}+\dfrac{1}{S_v^0}+\dfrac{2}{S_v^0}\right)\leq C_2 \eta$$
for each $t\geq 0$, with some constant $C_2$ independent of $\eta$ and $t$. It readily follows, by definition of $L_0$, that 
$$I_v(t)\leq C_2 S^0_v \eta=:C_3 \eta \underset{\eta \to 0}{\to} 0$$
uniformly in $t$. Using \eqref{Eq:Semiflow_h1}, the previous estimate on $I_v(t)$ and \eqref{Eq:Sh_estimate} we get
\begin{flalign*}
\int_t^\infty i_h(t,a)da&\leq \|i_h^\eta\|_{L^1}e^{-\mu_0 t}+\|\beta_v\|_{L^\infty}C_3 \eta\int_t^\infty \int_{a-t}^a s_h(t-a+s,s)e^{-\int_s^a \mu_h(\xi)d\xi}ds da \\
&\leq \eta e^{-\mu_0 t}+\|\beta_v\|_{L^\infty}C_3 \eta\int_t^\infty  \int_{a-t}^a (s_h^\eta+i_h^\eta+r_h^\eta)(a-t)e^{-\int_{a-t}^s \mu_h(\xi)d\xi} e^{-\int_s^a \mu_h(\xi)d\xi} ds da \\
&\leq \eta e^{-\mu_0 t}+\|\beta_v\|_{L^\infty} C_3 \eta \int_t^\infty (s_h^\eta+i_h^\eta+r_h^\eta)(a-t)e^{-\mu_0 t}t da \\
&\leq \eta e^{-\mu_0 t}+\|\beta_v\|_{L^\infty} C_3 \eta\left(\dfrac{\Lambda_h}{\mu_0}+3\eta\right)e^{-\mu_0 t}t \\
&\leq \eta e^{-\mu_0 t}+\|\beta_v\|_{L^\infty} C_3 \eta\left(\dfrac{\Lambda_h}{\mu_0}+3\eta\right)\dfrac{e^{-1}}{\mu_0} \\
&\leq \eta C_4
\end{flalign*}
for some constant $C_4$. Using \eqref{Eq:Semiflow_h2}, we can decompose $i_h$ as follows:
\begin{flalign*}
\|i_h(t,.)\|_{L^1}&=\int_0^t i_h(t,a)da+\int_t^\infty i_h(t,a)da \\
&\leq \int_0^t \int_{0}^{a}\beta_v(s)I_v(t-a+s)s_h(t-a+s,s)e^{-\int_s^a \mu_h(\xi)d\xi}ds da+\int_t^\infty i_h(t,a)da \\
&\leq  \|\beta_v\|_{L^\infty}C_3 \eta\int_0^t \int_{0}^a s^0_h(s) e^{-\int_s^a \mu_h(\xi)d\xi}ds da+\int_t^\infty i_h(t,a)da \\
&\leq \Lambda_h\|\beta_v\|_{L^\infty}C_3 \eta \int_0^t a e^{-\mu_0 a}da+\eta C_4 \\
&\leq \Lambda_h \|\beta_v\|_{L^\infty}C_3\eta\left(\dfrac{1-e^{-\mu_0 t}}{\mu_0^2}-\dfrac{t e^{-\mu_0 t}}{\mu_0}\right)+\eta C_4 \\
&\leq \Lambda_h \|\beta_v\|_{L^\infty} C_3 \dfrac{\eta}{\mu_0^2}+\eta C_4\leq C_5 \eta
\end{flalign*}
for some positive constant $C_5$, whence $\|i_h(t,.)\|_{L^1}\underset{\eta \to 0}{\longrightarrow} 0$ uniformly in $t$. For $r_h$, it suffices to use \eqref{Eq:Semiflow_h1}-\eqref{Eq:Semiflow_h2} to get:
\begin{flalign*}
\int_0^\infty r_h(t,a)da&\leq \left(\int_0^\infty r_h^\eta(a)da\right)e^{-\mu_0 t}+\dfrac{\|r_1\|_{L^\infty}C_5 \eta}{\mu_0}\left(1-e^{-\mu_0 t}\right)\\
&\leq\eta e^{-\mu_0 t}+\dfrac{\|r_1\|_{L^\infty}C_5 \eta}{\mu_0}\left(1-e^{-\mu_0 t}\right)\\
&\leq C_6 \eta
\end{flalign*}
for some positive constant $C_6$ so that $\|r_h(t,.)\|_{L^1}\underset{\eta \to 0}{\to}0$ uniformly in $t$. To prove the estimate on $S_v$, we remark that:
$$S_v'(t)\geq \Lambda_v-\|\beta_h\|_{L^\infty}C_5 \eta S_v(t)-\mu_v S_v(t)$$
leading to
\begin{flalign*}
S_v(t)\geq S_v^\eta e^{-(\|\beta_h\|_{L^\infty}C_5\eta+\mu_v)t}+\dfrac{\Lambda_v}{\|\beta_h\|_{L^\infty}C_5\eta +\mu_v}\left(1-e^{-(\|\beta_h\|_{L^\infty}C_5\eta+\mu_v)t}\right)
\end{flalign*}
whence
\begin{flalign*}
\dfrac{\Lambda_v}{\mu_v}-S_v(t)&\leq \dfrac{\Lambda_v}{\mu_v}-\dfrac{\Lambda_v}{\|\beta_h\|_{L^\infty}C_5 \eta +\mu_v}+e^{-(\|\beta_h\|_{L^\infty}C_5\eta+\mu_v)t}\left(\dfrac{\Lambda_v}{\|\beta_h\|_{L^\infty}C_5 \eta +\mu_v}-S^\eta_v\right)\\
&\leq \dfrac{\Lambda_v \|\beta_h\|_{L^\infty}C_5 \eta}{\mu_v(\|\beta_h\|_{L^\infty}C_5\eta+\mu_v)}+e^{-(\|\beta_h\|_{L^\infty}C_5\eta+\mu_v)t}\left(-\dfrac{\Lambda_v \|\beta_h\|_{L^\infty}C_5 \eta}{\mu_v(\|\beta_h\|_{L^\infty}C_5 \eta +\mu_v)}+S_v^0-S_v^\eta\right)\\
&\leq \dfrac{\Lambda_v \|\beta_h\|_{L^\infty}C_5 \eta}{\mu_v^2}+\eta.
\end{flalign*}
Moreover, using \eqref{Eq:Estimates_sol3} we know that:
\begin{flalign*}
S_v(t)-\dfrac{\Lambda_v}{\mu_v}&\leq \left(S_v^\eta+I_v^\eta -\dfrac{\Lambda_v}{\mu_v}\right)e^{-\mu_v t}\leq 2\eta
\end{flalign*}
so that 
$$\left\lvert S_v(t)-\dfrac{\Lambda_v}{\mu_v}\right\rvert\leq C_7 \eta$$
for some positive constant $C_7$ so that $|S_v(t)-S_v^0|\underset{\eta \to 0}{\to} 0$ uniformly in $t$. Finally, for the estimate on $s_h$, we see on one hand with \eqref{Eq:Estimates_sol1} that:
\begin{flalign*}
s_h(t,a)-s_h^0(a)\leq (s_h^\eta-s_h^0+i_h^\eta+r_h^\eta)(a-t)e^{-\int_{a-t}^a \mu_h(s)ds}\mathbf{1}_{\{a\geq t\}}.
\end{flalign*}
On the other hand, we have, by using \eqref{Eq:Semiflow_h1}-\eqref{Eq:Semiflow_h2}:
\begin{equation*}
s_h(t,a)\geq \begin{cases}
s_h^\eta(a-t)e^{-\int_{a-t}^a (\beta_v(s)C_3 \eta+\mu_h(s))ds} & \text{ if } a\geq t \\
\Lambda_h e^{-\int_0^a (\beta_v(s)C_3 \eta +\mu_h(s))ds} & \text{ if } a<t
\end{cases}
\end{equation*}
Suppose that $a<t$, then the two latter equations lead to
\begin{flalign*}
0\geq s_h(t,a)-s^0_h(a)\geq \Lambda_h e^{-\int_0^a \mu_h(s)ds}\left(e^{-\int_0^a \beta_v(s)C_3 \eta ds}-1\right).
\end{flalign*}
It follows that
\begin{flalign*}
\int_0^t |s_h(t,a)-s^0_h(a)|da&\leq \Lambda_h \int_0^t e^{-\mu_0 a}\left(1-e^{-\|\beta_v\|_{L^\infty}C_3 \eta a}\right)da.
\end{flalign*}
The function $a\longmapsto e^{-\mu_0 a}\left(1-e^{-\|\beta_v\|_{L^\infty}C_3 \eta a}\right)$ is bounded above by the function $a\longmapsto e^{-\mu_0 a}$ which is integrable on $\R_+$ and independent of $\eta$. Moreover, we see that 
$$e^{-\mu_0 a}\left(1-e^{-\|\beta_v\|_{L^\infty}C_3 \eta a}\right)\underset{\eta \to 0}{\to}0$$
for each $a\geq 0$. Consequently, by using Lebesgue's dominated convergence theorem, we have 
$$\int_0^t |s_h(t,a)-s^0_h(a)|da\underset{\eta \to 0}{\to} 0$$
uniformly in $t$. Now, suppose that $a\geq t$, then 
\begin{flalign*}
s_h(t,a)-s^0_h(a)\geq s_h^0(a-t)e^{-\int_{a-t}^a \mu_h(s)ds}\left(e^{-\int_{a-t}^a \beta_v(s)C_3 \eta}-1\right)+\left(s^\eta_h(a-t)-s^0_h(a-t)\right)e^{-\int_{a-t}^a (\beta_v(s)C_3\eta+\mu_h(s))ds}.
\end{flalign*}
It follows that
\begin{flalign*}
\int_t^\infty |s_h(t,a)-s^0_h(a)|da&\leq \left(\|s_h^\eta-s^0_h\|_{L^1}+\|i_h^\eta\|_{L^1}+\|r_h^\eta\|_{L^1}\right)e^{-\mu_0 t}+\int_t^\infty s^0_h(a-t)e^{-\int_{a-t}^a \mu_h(s)ds}\left(1-e^{-\int_{a-t}^a \beta_v(s)C_3 \eta}\right)da\\
&\leq 3\eta e^{-\mu_0 t}+\Lambda_h\int_t^\infty e^{-\int_0^a \mu_h(s)ds}\left(1-e^{-\|\beta_v\|_{L^\infty}C_3\eta t}\right)da \\
&\leq 3\eta e^{-\mu_0 t}+\Lambda_h\int_t^\infty e^{-\int_0^a \mu_h(s)ds}\left(1-e^{-\|\beta_v\|_{L^\infty}C_3\eta a}\right)da\underset{\eta \to 0}{\to }0
\end{flalign*}
uniformly in $t$ by using Lebesgue's theorem once again. Gathering all previous estimates, we obtain:
$$\lim_{\eta \to 0}\|\Phi_t(\widehat{z})-E_0\|_{\X}=0$$
uniformly in $t$. Consequently, considering small enough $\eta$ proves the stability of $E_0$ whenever $\RR_0=1$.
\end{proof}

The local stability obtained in Proposition \ref{Prop:LAS_E0} for $\RR_0<1$ and Theorem \ref{Thm:stab_R0=1} for $\RR_0=1$, combined with the attractiveness proved in Proposition \ref{Prop:L0_lyap} lead to the following result.
\begin{theorem}
\label{gas_dfe}
The parasite-free equilibrium $E_0$ is globally asymptotically stable in $\X_+$ whenever $\RR_0\leq 1$.
\end{theorem}
It is worth noting that in \cite{magal2010}, the stability was handled by proving the existence of a global attractor that is stable provided that it attracts all compact subsets of a neighbourhood of itself (see \cite[Thm 3.3.2, p. 38]{Hale88} or \cite[Thm 2.39, p. 47]{SmithThieme2011}), then reducing this attractor to the corresponding equilibrium point (parasite-free or endemic equilibrium) proves both the stability and the attractiveness.

\subsection{On the endemic equilibrium}

The search of endemic equilibrium is hard to investigate in the general case. First we handle the case without reinfection, that is we suppose:
\begin{assumption}\label{Assump:r2=0}
We assume that $r_2\equiv 0$.
\end{assumption}

\begin{theorem}\label{Thm:equ-nor2}
Under Assumption \ref{Assump:r2=0} the model \eqref{system} has a unique equilibrium.
\end{theorem}

\begin{proof}
A steady state $E^*=(s_h^*,i_h^*,r_h^*,S_v^*,I_v^*)$ of the model \eqref{system} is given by the following equations:
\begin{equation}\label{Eq:system_e*}
\left\{
\begin{array}{ll}
(s_h^*)'(a) & =  -\beta_v(a) s_h^*(a) I_v^* - \mu_h(a) s_h^*(a)\\ 
(i_h^*)'(a) & = \beta_v(a) s_h^*(a) I_v^* - (\mu_h(a)+r_1(a)+\delta(a)) i_h^*(a)\\
(r_h^*)'(a) & = r_1(a) i_h^*(a) - \mu_h(a) r_h^*(a) \\ 
0 & = \Lambda_v  - S_v^* \int_0^{+\infty} \beta_h(a) i_h^*(a)da - \mu_v S_v^* \\
0 & = S_v^* \int_0^{+\infty} \beta_h(a) i_h^*(a)da - \mu_v I_v^* \\
\end{array} \right.
\end{equation}
It follows that
$$S_v^*=\dfrac{\Lambda_v}{\mu_v+\int_0^\infty \beta_h(a)i_h^*(a)da}, \qquad I_v^*=\dfrac{\Lambda_v \int_0^\infty \beta_h(a)i_h^*(a)da}{\mu_v(\mu_v+\int_0^\infty \beta_h(a)i_h^*(a)da)}$$
and
$$s_h^*(a)=\Lambda_h e^{-\int_0^a (\beta_v(s) I_v^*+\mu_h(s))ds}$$
$$i_h^*(a)=\Lambda_h I_v^*\int_0^a \beta_v(s)e^{-\int_0^s(\beta_v(\xi) I_v^*+\mu_h(\xi))d\xi}e^{-\int_s^a (\mu_h(\xi)+r_1(\xi)+\delta(\xi))d\xi}ds.$$
Multiplying the latter equation by $\beta_h(a)$ and integrating w.r.t. $a$ we get
$$1=\left(\dfrac{\Lambda_h \Lambda_v}{\mu_v\left(\mu_v+\int_0^\infty \beta_h(z)i_h^*(z)dz\right)}\right)\int_0^\infty \beta_h(a)\int_0^a \beta_v(s)e^{-\int_0^s\frac{\Lambda_v\beta_v(\xi)\int_0^\infty \beta_h(z)i_h^*(z)dz}{\mu_v(\mu_v+\int_0^\infty \beta_h(z)i_h^*(z)dz)}+\mu_h(\xi))d\xi} e^{-\int_s^a (\mu_h(\xi)+r_1(\xi)+\delta(\xi))d\xi}ds da.$$
Writing the latter equation as $1=f(\int_0^\infty \beta_h(z)i_h^*(z)dz)$, where $f:\R_+\to \R_+$ is a continuous decreasing function with $f(0)=\RR_0>1$, we deduce that there exists a unique positive solution to $f(x)=1$ which is $\int_0^\infty \beta_h(z)i_h^*(z)dz$. It follows that there exists a unique solution to \eqref{Eq:system_e*} and then a unique equilibrium to \eqref{system}.
\end{proof}

Now, we assume that Assumption \ref{Assump:r2=0} does not hold but we neglect the age-dependency of the parameters:
\begin{assumption}\label{Assump:no_age}
We assume that the parameters $\beta_h$, $\beta_v$, $r_1$, $r_2$ and $\mu_h$ are constants.
\end{assumption}

\begin{theorem}\label{Thm:equ-noage}
Under Assumption \ref{Assump:no_age}, the model \eqref{Eq:EDO} has a unique equilibrium.
\end{theorem}

\begin{proof}
Since Assumption \ref{Assump:no_age} holds, then \eqref{system} rewrites as \eqref{Eq:EDO} after integration w.r.t. $a$. The formula of the basic reproduction number simply becomes:
\[\RR _0 = \sqrt{\frac{\beta_h S_v^0}{\mu_h+\delta+r_1} \times \frac{\beta_v S_h^0}{\mu_v}} \]
with $S_v^0 = \frac{\Lambda_v}{\mu_v}$ and $S_h^0 = \frac{\Lambda_h}{\mu_h}$. The search of endemic equilibria brings us to
\begin{equation}\label{Eq:equ_EDO}
\left\{ \begin{array}{lll}
\Lambda_h+r_2 R_h^* & =  (\beta_v I_v^*+\mu_h) S_h^* \\ 
\beta_v I_v^* S_h^* & = (\mu_h+\delta+r_1)I_h^*\\
r_1 I_h^* & =(r_2+\mu_h) R_h^* \\ 
\Lambda_v & = \beta_h S_v^* I_h^* + \mu_v S_v^* \\
\beta_h S_v^*I_h^* & = \mu_v I_v^*.
\end{array}
\right. 
\end{equation}
Then we get
\begin{align*}
\mu_v I_v^* & = \dfrac{ \beta_h S_v^* \beta_v I_v^* S_h^*}{\mu_h+\delta+r_1}\\
			&= \left(\dfrac{ \beta_h S_v^* \beta_v I_v^*}{\mu_h+\delta+r_1}\right)\times \left(\dfrac{\Lambda_h+r_2R_h^*}{\beta_v I_v^*+\mu_h}\right) \\
			&= \left(\dfrac{\beta_v I_v^*}{(\mu_h+\delta+r_1)(\beta_v I_v^*+\mu_h)}\right)\times \beta_h S_v^* \left(\Lambda_h+\dfrac{r_2r_1 I_h^*}{r_2+\mu_h}\right) \\
			&= \left(\dfrac{\beta_v I_v^*}{(\mu_h+\delta+r_1)(\beta_v I_v^*+\mu_h)}\right)\times \left(\Lambda_h\beta_h\left(\dfrac{\Lambda_v}{\mu_v}-I_v^*\right)+\dfrac{r_2r_1 \mu_vI_v^*}{r_2+\mu_h}\right).
\end{align*}
whence either
\[I_v^* = 0 \quad  \mbox{ or } \quad \mu_v (\beta_v I_v^* + \mu_h)  = \frac{\Lambda_h \beta_h \beta_v}{\mu_h+\delta+r_1} \left(\frac{\Lambda_v}{\mu_v}-I_v^*\right)   + \frac{r_1 r_2 \beta_v}{(r_2+\mu_h)(\mu_h+\delta+r_1)} \mu_v I_v^*.\]
leading, if $I_v^*\neq 0$, to 
\begin{equation*}
\displaystyle I_v^* = c \left(\RR_0^2  - 1\right)
\label{ee}
\end{equation*}
where
$$\displaystyle c = \frac{\mu_h \mu_v(\mu_h+\delta+r_1)(\mu_h+r_2)}{\beta_v \mu_v(\mu_h+\delta+r_1)(\mu_h+r_2) + \Lambda_h\beta_v \beta_h(\mu_h+r_2)-\beta_v\mu_vr_1r_2} > 0$$
if $\RR_0>1$. It follows that
$$S_v^*=\dfrac{\Lambda_v}{\mu_v}-I_v^*, \qquad  I_h^*=\dfrac{\mu_v I_v^*}{\beta_h S_v^*}, \qquad R_h^*=\dfrac{r_1 I_h^*}{r_2+\mu_h}, \qquad S_h^*=\dfrac{\Lambda_h+r_2R_h^*}{\beta_v I_v^*+\mu_h}.$$

\end{proof}

We end this section with a proof of local asymptotic stability of the endemic equilibrium in a specific case.

\begin{theorem}
Under Assumptions \ref{Assump:r2=0} and \ref{Assump:no_age}, the unique endemic equilibrium $E^*$ of \eqref{Eq:EDO} is LAS.
\end{theorem}

\begin{proof}
The existence and uniqueness of $E^*$ is shown in Theorems \ref{Thm:equ-nor2}-\ref{Thm:equ-noage}, and the components are explicited in the proof of Theorem \ref{Thm:equ-noage} (with $r_2\equiv 0$).
Linearising the system \eqref{Eq:EDO} around $E^*$, we get
\begin{equation*}
\left\{ \begin{array}{lll}
S_h'(t)& =-\beta_v I_v^*S_h(t)-\beta_v I_v(t)S_h^*-\mu_h S_h(t)  \\ 
I_h'(t)& = \beta_v I_v^*S_h(t)+\beta_vI_v(t)S_h^* - (\mu_h+\delta+r_1)I_h(t)\\
R_h'(t) & = r_1 I_h(t) - \mu_h R_h(t) \\ 
S_v'(t) & = - \beta_h S_v^* I_h(t)-\beta_h S_v(t)I_h^* - \mu_v S_v(t) \\
I_v'(t) & = \beta_h S_v^*I_h(t)+\beta_h S_v(t)I_h^*- \mu_v I_v(t).
\end{array}
\right.
\end{equation*}
The eigenvalues $\lambda\in \C$ then satisfy the following system:
\begin{equation}\label{Eq:EDO-lin}
\left\{ \begin{array}{lll}
(\lambda+\mu_h+\beta_v I_v^*)S_h&=&-\beta_v I_v S_h^* \\ 
(\lambda+\mu_h+\delta+r_1)I_h&=&\beta_v I_v^*S_h+\beta_vI_vS_h^* \\
(\lambda+\mu_h)R_h&=&r_1I_h\\ 
(\lambda+\mu_v+\beta_h I_h^*)S_v&=&-\beta_h S_v^*I_h \\
(\lambda+\mu_v)I_v&=&\beta_h S_v^*I_h+\beta_h S_vI_h^*.
\end{array}
\right.
\end{equation}
The first two equations of \eqref{Eq:EDO-lin} lead to
$$(\lambda+\mu_h+\delta+r_1)I_h=\beta_vI_vS_h^*\left(1-\dfrac{\beta_v I_v^*}{\lambda+\mu_h+\beta_vI_v^*}\right)$$
while the last two equations of \eqref{Eq:EDO-lin} give
$$(\lambda+\mu_v)(S_v+I_v)=0.$$
If $\lambda=-\mu_v<0$ then $\lambda$ is a negative eigenvalue, otherwise $I_v=-S_v$. From the fourth equation of \eqref{Eq:EDO-lin} we deduce that
$$(\lambda+\mu_h+\delta+r_1)I_h=\beta_vS_h^*\left(1-\dfrac{\beta_v I_v^*}{\lambda+\mu_h+\beta_vI_v^*}\right)\dfrac{\beta_hS_v^*I_h}{\lambda+\mu_v+\beta_hI_h^*}.$$
If $I_h=0$ then $(S_h,I_h,R_h,S_v,I_v)=0$ and $\lambda$ is not an eigenvalue. Otherwise we get $p(\lambda)=0$ where
$$p(\lambda)=(\lambda+\mu_h+\delta+r_1)(\lambda+\mu_v+\beta_hI_h^*)(\lambda+\mu_h+\beta_v I_v^*)-\beta_h\beta_v S_h^* S_v^*(\lambda+\mu_h)$$
which rewrites as the third degree polynomial function $p(\lambda)=a_3\lambda^3+a_2\lambda^2+a_1\lambda+a_0$ with
\begin{equation}
\left\{
\begin{array}{rcl}
a_3&=&1 \\
a_2&=&(\mu_h+\delta+r_1)+(\mu_v+\beta_hI_h^*)+(\mu_h+\beta_vI_v^*) \\
a_1&=&(\mu_h+\delta+r_1)(\mu_v+\beta_hI_h^*)+(\mu_h+\delta+r_1)(\mu_h+\beta_vI_v^*)+(\mu_v+\beta_hI_h^*)(\mu_h+\beta_vI_v^*)-\beta_h\beta_vS_h^*S_v^* \\
a_0&=&(\mu_h+\delta+r_1)(\mu_v+\beta_hI_h^*)(\mu_h+\beta_vI_v^*)-\mu_h\beta_h\beta_vS_h^*S_v^*.
\end{array}
\right.
\end{equation}
We readily see that $a_3>0$ and $a_2>0$. From \eqref{Eq:equ_EDO} we deduce that
$$\beta_h\beta_vS_v^*S_h^*=\mu_v(\mu_h+\delta+r_1)$$
whence
$$a_1=\beta_hI_h^*(\mu_h+\delta+r_1)+(\mu_h+\delta+r_1)(\mu_h+\beta_vI_v^*)+(\mu_v+\beta_hI_h^*)(\mu_h+\beta_vI_v^*)>0$$
and
$$a_0=(\mu_h+\delta+r_1)\left((\mu_v+\beta_hI_h^*)(\mu_h+\beta_vI_v^*)-\mu_h\mu_v\right)>0.$$
To find the roots of $p$ with non-negative real part, we use the criterion of Routh-Hurwitz. We compute
$$b_1=\dfrac{-1}{a_1}\left|\begin{pmatrix}
a_3 & a_1 \\
a_2 & a_0
\end{pmatrix}\right|=a_2-\dfrac{a_3a_0}{a_1}.$$
It follows that $b_1>0 \Longleftrightarrow a_2a_1>a_0$.
We see that
$$a_2a_1\geq (\mu_h+\delta+r_1)(\mu_v+\beta_hI_h^*)(\mu_h+\beta_vI_v^*)>a_0$$
whence $b_1>0$. Finally we compute
$$c_0=\dfrac{-1}{b_1}\left|\begin{pmatrix}
a_2 & a_0 \\
b_1 & 0
\end{pmatrix}\right|=a_0>0.$$
We conclude by the criterion that there exists no roots of $p$ with non-negative real part, hence $E^*$ is LAS.
\end{proof}

\section{Numerical simulations}

In this section, we show numerical simulations to validate the previous theoretical results and provide insights into the stability of the endemic equilibrium.

\subsection{Numerical scheme}

The numerical method for the model \eqref{system} is based on finite volumes methods. We introduce $\Delta t$ and $\Delta a$, the constant time and age steps respectively. Denoting $\lceil a \rceil$ the upper integer part of $a$, we define $N_a = \lceil a_{\max}/\Delta a \rceil$ and $N_T = \lceil T/\Delta t \rceil$ where $a_{\max}$ and $T$ are respectively the maximal time and maximal age of the numerical simulation.

Let us introduce the points $a_{j+1/2} = j\Delta a$, for $0 \leq j \leq N_a$ and the cells $K_j = [a_{j-1/2}, a_{j+1/2}[$, the centers of the cells $a_j = (j - 1/2)\Delta a$, for $1 \leq j \leq N_a$ and $t^n = n\Delta t$ for $0 \leq n \leq N_T$. For $1 \leq j \leq N_a$ and $0 \leq n \leq N_T$, we denote by $s^n_{h,j}, i^n_{h,j}$ and $r^n_{h,j}$ respectively the approximation of the average of $s_h(t^n, \cdot), i_h(t^n, \cdot), r_h(t^n, \cdot)$ on the cell $K_j$, namely

\[
s^0_{h,j} = \frac{1}{\Delta a} \int_{K_j} s_{h,0}(a)da, \quad
i^0_{h,j} = \frac{1}{\Delta a} \int_{K_j} i_{h,0}(a)da, \quad
r^0_{h,j} = \frac{1}{\Delta a} \int_{K_j} r_{h,0}(a)da
\]

and

\[
s^n_{h,j} \approx \frac{1}{\Delta a} \int_{K_j} s_h(t^n, a)da, \quad
i^n_{h,j} \approx \frac{1}{\Delta a} \int_{K_j} i_h(t^n, a)da, \quad
r^n_{h,j} \approx \frac{1}{\Delta a} \int_{K_j} r_h(t^n, a)da
\]

for every $n \geq 1$. We also define

\[
S^0_v = S_{v,0}, \quad I^0_v = I_{v,0}.
\]

For $1 \leq j \leq N_a$, we let

\[
\beta_{v,j} = \frac{1}{\Delta a} \int_{K_j} \beta_v(a)da, \quad
\beta_{h,j} = \frac{1}{\Delta a} \int_{K_j} \beta_h(a)da, \quad
\mu_{h,j} = \frac{1}{\Delta a} \int_{K_j} \mu_h(a)da,
\]

\[
\delta_j = \frac{1}{\Delta a} \int_{K_j} \delta(a)da, \quad
r_{1,j} = \frac{1}{\Delta a} \int_{K_j} r_1(a)da, \quad
r_{2,j} = \frac{1}{\Delta a} \int_{K_j} r_2(a)da.
\]

The numerical scheme is :

\[
\begin{cases}
S^{n+1}_v = \dfrac{S^n_v + \Lambda_v \Delta t}{1 + \Delta t \Delta a \sum_{j=1}^{N_a} \beta_{h,j} i^n_{h,j} + \mu_v \Delta t} \vspace{0.1cm} \\

I^{n+1}_v = \dfrac{I^n_v + \Delta t \Delta a \sum_{j=1}^{N_a} \beta_{h,j} i^n_{h,j} S^{n+1}_v}{1 + \Delta t \mu_v} \vspace{0.1cm} \\

s^{n+1}_{h,j} = \dfrac{s^n_{h,j} (1 - \frac{\Delta t}{\Delta a}) + \frac{\Delta t}{\Delta a} s^n_{h,j-1} +  r_{2,j} r^n_{h,j}\Delta t}{1 + \Delta t \beta_{v,j} I^{n+1}_v + \Delta t \mu_{h,j}} \vspace{0.1cm} \\

i^{n+1}_{h,j} = \dfrac{i^n_{h,j} (1 - \frac{\Delta t}{\Delta a}) + \frac{\Delta t}{\Delta a} i^n_{h,j-1} + \Delta t \beta_{v,j} I^{n+1}_v s^{n+1}_{h,j}}{1 + \Delta t (\mu_{h,j} + \delta_j + r_{1,j})} \vspace{0.1cm} \\

r^{n+1}_{h,j} = \dfrac{r^n_{h,j} (1 - \frac{\Delta t}{\Delta a}) + \frac{\Delta t}{\Delta a} r^n_{h,j-1} +  r_{1,j} i^{n+1}_{h,j}\Delta t}{1 + \Delta t (r_{2,j} + \mu_{h,j})}
\end{cases}
\]

for every $n \geq 0$ and $j \geq 1$ with the boundary conditions :

\[
\begin{cases}
s^{n+1}_{h,0} = \Lambda_h \\
i^{n+1}_{h,0} = 0 \\
r^{n+1}_{h,0} = 0
\end{cases}
\]

\subsection{Parameters}

For the numerical simulations, we choose the parameters summarized in the Table \ref{tab}. The function $\delta$ is an interpolation of the data retrieved in \cite{Streatfield2014} for Burkina Faso, Ouagadougou case.  We consider $\Lambda_h=10^{6}\times 3.37 \cdot 10^{-2}$ with $3.37\cdot 10^{-2}$ the 2023 crude birth rate in Cameroon \cite{UNDESA2024} and a total human population of $10^6$ inhabitants. We assume that the initial human population is at the parasite-free equilibrium, that is 
$$s_{h,0}(a)=\Lambda_h e^{-\int_0^a \mu_h(s)ds}, \quad i_{h,0}\equiv 0, \quad r_{h,0}\equiv 0.$$
We consider logistic functions $r_1, r_2$ such that $r_1(0)=0.5, \lim_{a\to +\infty}r_1(a)=6$ corresponding to the time recovery estimated in \cite{Bekessy76,Molineaux80} and where $r_2(0)=1, \lim_{a\to +\infty}r_2(a)=1/5$ corresponding to a period of immunity between 1 and 5 years as assumed in \cite{chitnis08}. We decompose the force of infection $\lambda_v(t)$ as
$$\lambda_v(t)=\frac{\theta \int_0^\infty G(a)i_h(t,a)da}{\Lambda_h \int_0^\infty e^{-\int_0^s \mu_h(\xi)d\xi}ds}$$
where:
\begin{enumerate}[label=\tiny $\bullet$]
\item $\dfrac{i_h(t,a)}{\Lambda_h \int_0^\infty e^{-\int_0^s \mu_h(\xi)d\xi}ds}$ is the proportion of infected human of age $a$ at time $t$ (supposing a constant population given by the parasite-free equilibrium),
\item $\theta$ is the human feeding rate, that is the expected number of bites on humans per mosquito per year, taken as $\theta=0.6\times 365$ \cite{chitnis06},
\item $G(a)$ is the probability that a human of age $a$ transmits the infection to a susceptible biting mosquito which is considered as 
$$G(a)=7.1 \cdot  10^{-2} (G_0(a))^{3.02\cdot 10^{-1}}$$
(see \cite{Bradley2018}), with $G_0(a)$ the mean number of gametocytes for a human of age $a$, estimated in \cite{Churcher2013}. 
\end{enumerate}
Similarly, we decompose the force of infection $\lambda_h(t,a)$ as
$$\lambda_h(t,a)=\dfrac{\theta \tilde{\beta_v}G_0(a)I_v(t)}{\|G_0\|_{L^\infty}\Lambda_h \int_0^\infty e^{-\int_0^s \mu_h(\xi)d\xi}ds}$$
where $\frac{\tilde{\beta_v}G_0(a)}{\|G_0\|_{L^\infty}}$ is the probability that a susceptible human of age $a$ gets infected from a mosquito bite, with $G_0$ describing the vulnerability for humans to be infected and normalized such that the maximum is $\tilde{\beta_v}$ which is estimated in \cite{chitnis08}.

The parameter $\Lambda_v$ is adjusted to obtain different values of $\RR_0$. As in \cite{niger08} we assume a mosquito life expectancy of $20$ days leading to $\mu_v=365/20$ year$^{-1}$, despite the high variability of this parameter in practice (see \cite{Richard24} and the references therein for a discussion on this topic). Finally, we consider a susceptible mosquito population at the parasite-free equilibrium, $S_{v,0}=\frac{\Lambda_v}{\mu_v}$, while $I_{v,0}$ is allowed to vary.

\begin{table}[!hbtp]
\begin{center}
\begin{tabular}{|l|l|l|l|}
\hline
\textbf{Parameters} & \textbf{Value} & \textbf{Unit}& \textbf{References} \\
\hline 
$\Lambda_h$ & $3.37\cdot 10^{4}$ & year$^{-1}$  &\cite{UNDESA2024} \\
\hline 
$\mu_h(a)$ & $5.8\left(2\cdot 10^{-3}+9\cdot 10^{-2}e^{-2.1a}+10^{-4}e^{9\cdot 10^{-2}a}\right)$ & year$^{-1}$  &\cite{Zhuolin23} \\
\hline
$\delta(a)$ & $10^{-3}\left(6.58\cdot 10^{-1}e^{-3.405\cdot 10^{-1}a}+4\cdot 10^{-4}e^{4.44\cdot 10^{-2}a}\right)$ & year$^{-1}$  &\cite{Streatfield2014} \\
\hline
$r_1(a)$ & $6\cdot(1+11e^{-0.05a})^{-1}$  & year$^{-1}$ & \cite{Bekessy76,Molineaux80}\\
\hline
$r_2(a)$ & $(5-4e^{-0.05a})^{-1}$ & year$^{-1}$  & \cite{chitnis08} \\
\hline
$\Lambda_v$ & Varies  &  year$^{-1}$ & Assumed \\
\hline
$\mu_v$ & $18.25$ & year$^{-1}$& \cite{niger08}\\
\hline
$\theta$ & $219$ & year$^{-1}$& \cite{chitnis06}\\
\hline 
$G_0(a)$ & $22.7a e^{-0.0934a}$ & no unit & \cite{Churcher2013}\\
\hline 
$G(a)$ & $7.1 \cdot  10^{-2} (G_0(a))^{3.02\cdot 10^{-1}}$ & no unit & \cite{Bradley2018}\\
\hline
$\beta_h(a)$ & $\theta \cdot G(a) \cdot (\Lambda_h \int_0^\infty e^{-\int_0^s \mu_h(\xi)d\xi}ds)^{-1}$   & year$^{-1}$  & \cite{Richard24}  \\
\hline
$\tilde{\beta_v}$ & $0.022$ & no unit & \cite{chitnis08} \\
\hline
$\beta_v(a)$ & $\theta\cdot \tilde{\beta_v}G_0(a)\cdot \|G_0\|_{L^\infty}^{-1}\cdot (\Lambda_h \int_0^\infty e^{-\int_0^s \mu_h(\xi)d\xi}ds)^{-1}$ & year$^{-1}$ & Assumed \\
\hline  
\end{tabular}
\end{center}
\caption{Parameters for the numerical simulations of \eqref{system}}
\label{tab}
\end{table}

\subsection{Simulations}

In the Figure \ref{Fig:R0>1}, we set $\Lambda_v=5\cdot 10^6$ leading to $\RR_0\approx 1.52$ and we consider four different initial conditions: $I_{v,0}\in \{10^3, 10^4, 10^5, 2\cdot 10^5\}$. Each solution converge to the same endemic equilibrium. In the Figure \ref{Fig:R0<1}, we set $\Lambda_v=1.5\cdot 10^6$ leading to $\RR_0\approx 0.83$ with the same four initial conditions. We plot in log scale to see the decreasing of the solutions to the parasite-free equilibrium.

\begin{figure}[!hbtp]
\begin{center}
\begin{tabular}{cc}
\includegraphics[scale=0.5]{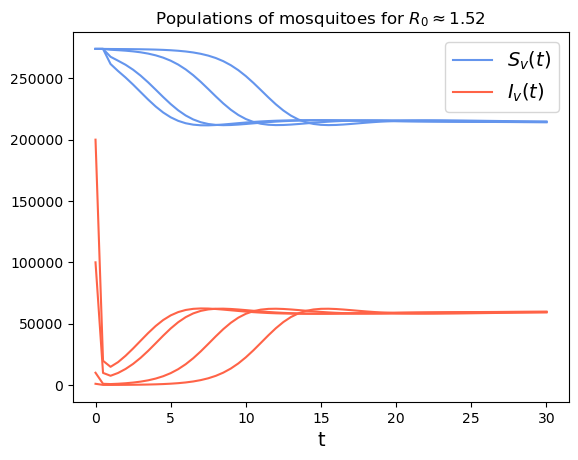} & \includegraphics[scale=0.5]{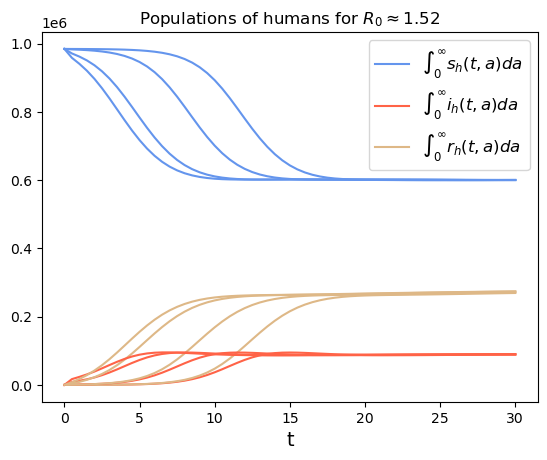}
\end{tabular}
\caption{Convergence to an endemic equilibrium when $\RR_0>1$ for different initial conditions}
\label{Fig:R0>1}
\end{center}
\end{figure}

\begin{figure}[!hbtp]
\begin{center}
\begin{tabular}{cc}
\includegraphics[scale=0.5]{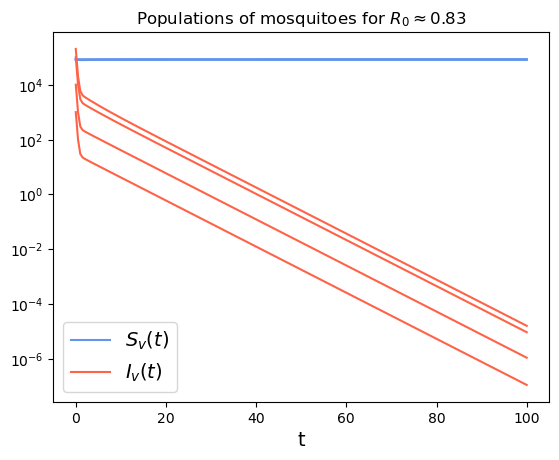} & \includegraphics[scale=0.5]{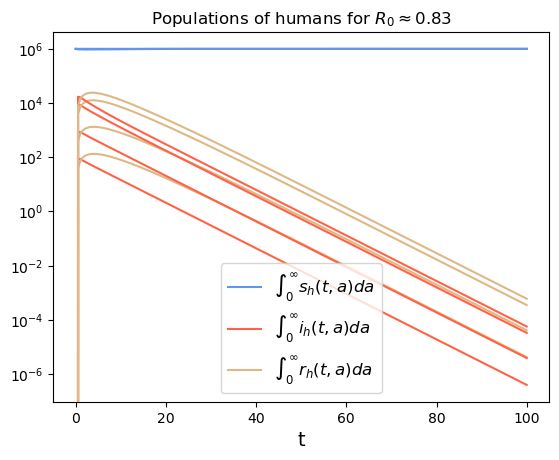}
\end{tabular}
\caption{Convergence to the parasite-free equilibrium when $\RR_0<1$ for different initial conditions}
\label{Fig:R0<1}
\end{center}
\end{figure}

\bibliographystyle{plain} 
\bibliography{bibliomalaria}

\end{document}